%% file: main.tex
\documentclass{article}
\pdfoutput=1
\usepackage{amsmath, amsthm, amsfonts,stmaryrd,amssymb}
\usepackage[T1]{fontenc}
\usepackage[utf8]{inputenc}
\usepackage[english]{babel}
\usepackage{mystyle}
\usepackage{notations}
\usepackage{dsfont}
\usepackage{url}
\usepackage[shortlabels]{enumitem}
\usepackage{hyperref}
\usepackage{multirow}
\usepackage{booktabs}

\newtheorem{theorem}{Theorem}[section]
\newtheorem*{theorem*}{Theorem}
\newtheorem{prop}[theorem]{Proposition}
\newtheorem{lemma}[theorem]{Lemma}
\newtheorem{corollary}[theorem]{Corollary}
\newtheorem{def/prop}{Definition/Proposition}

\theoremstyle{definition}
\newtheorem{definition}[theorem]{Definition}

\newtheorem{assumption}{Assumption}[section]
\newtheorem{example}{Example}[section]
\newtheorem{notat}{Notation}[section]

\theoremstyle{remark}
\newtheorem{remark}{Remark}[section]

\numberwithin{equation}{section}

\AtBeginDocument{}

\renewcommand{\epsilon}{\varepsilon}

\def\R{{\mathbb R}}
\newcommand{\eps}{\varepsilon}
\let\div\relax
\DeclareMathOperator{\div}{div}
\newcommand{\Rd}{{\R^d}}

\newcommand{\bracket}[1]{\langle #1 \rangle}

\newcommand{\abs}[1]{{\lvert #1 \rvert}}

\newcommand{\xb}{x_*}
\newcommand{\norm}[1]{\lVert #1 \rVert}
\newcommand{\normlinf}[1]{\norm{ #1 }_{L^{\infty}}}

\newcommand{\ibar}{{\bar\imath}}
\newcommand{\jbar}{{\bar\jmath}}
\newcommand{\bari}{\ibar}

\newcommand{\bark}{{\bar k}}
\newcommand{\kbar}{\bark}
\newcommand{\ib}{{\bar\imath}}
\newcommand{\jb}{{\bar\jmath}}
\newcommand{\kb}{{\bar k}}
\newcommand{\lb}{{\bar\ell}}
\newcommand{\mb}{{\bar m}}
\newcommand{\nb}{{\bar n}}
\newcommand{\sbar}{{\bar s}}
\newcommand{\tb}{{\bar t}}

\newcommand{\gt}{\tilde{g}}
\newcommand{\mt}{\tilde{m}}
\newcommand{\Gammat}{\tilde{\Gamma}}

\newcommand{\Gammah}{\Gammat}
\newcommand{\hGamma}{\Gammat}
\newcommand{\hnabla}{\tilde{\nabla}}
\newcommand{\nablat}{\hnabla}
\newcommand{\hR}{\tilde{R}}
\newcommand{\Rt}{\hR}
\newcommand{\Deltat}{{\tilde\Delta}}
\newcommand{\Gt}{\tilde{G}}

\newcommand{\calR}{\mathcal{R}}
\newcommand{\calU}{\mathcal{U}}

\renewcommand{\rho}{\varrho}

\makeatletter
\def\namedlabel#1#2{\begingroup
    #2%
    \def\@currentlabel{#2}%
    \phantomsection\label{#1}\endgroup
}
\title{A geometric Laplace method}
\author{
    Flavien Léger\thanks{INRIA Paris (\texttt{flavien.leger@inria.fr}).} 
  $\,$ and 
  François-Xavier Vialard\thanks{Université Gustave Eiffel, LIGM, CNRS and INRIA (\texttt{fxvialard@normalesup.org}).}
}

\begin{document}

\maketitle

\input{abstract}
\tableofcontents
%\todototoc
%\listoftodos

\input{intro}
%\newpage
\input{embeddings}

%\newpage
\input{geometric_laplace}

%\newpage
\input{quantitative_remainder}

%\newpage
\input{applications}

\input{ack}

\bibliographystyle{apalike}
\bibliography{articles.bib}

\end{document}

%% file: abstract.tex
\begin{abstract}
  A classical tool for approximating integrals is the Laplace method. The first-order, as well as the higher-order Laplace formula is most often written in coordinates without any geometrical interpretation. In this article, motivated by a situation arising, among others, in optimal transport, we give a geometric formulation of the first-order term of the Laplace method. The central tool is the Kim--McCann Riemannian metric which was introduced in the field of optimal transportation. Our main result expresses the first-order term  with standard geometric objects such as volume forms, Laplacians, covariant derivatives and scalar curvatures of two different metrics arising naturally in the Kim--McCann framework. Passing by, we give an explicitly quantified version of the Laplace formula, as well as examples of applications.
\end{abstract}

%% file: intro.tex
% !TEX root = main.tex

\section{Introduction}
\label{sec:introduction}

In its simplest form, the Laplace method consists in studying the behavior of the integral $\int_\Rd e^{-u(x)/\eps} \,r(x)dx$ as $\eps\to 0^+$. If $u\colon \R^d \to \R$ is sufficiently smooth and has a non-degenerate minimum at a unique point $\xb$, the Laplace method gives at first order in $\eps$, for instance when $r(x) = 1$,~\cite[Section 2]{ShunMcCullagh}
\begin{multline}\label{EqSimpleLaplace}
  \int_{\Rd} \frac{e^{-u(x)/\eps}}{(2\pi\eps)^{d/2}} \,dx = \frac{e^{-u/\eps}}{\sqrt{\det[u_{ij}]}}\bigg(1 
  + 
    \frac{\eps}{8} u_{ijk}u_{\ell mn}u^{ij}u^{k\ell}u^{mn}\\
  +\frac{\eps}{12}  u_{ijk}u_{\ell mn} u^{i\ell} u^{jm} u^{kn} - \frac{\eps}{8}  u_{ijk\ell}u^{ij}u^{k\ell}+ O(\eps^2)\bigg)\,.
\end{multline}
In this formula, $u_{ij}\coloneqq\partial_{ij}u$, $u_{ijk}\coloneqq\partial_{ijk}u$, etc, and we denote by $u^{ij}$ the inverse matrix of $u_{ij}$. All the quantities involving $u$ on the right-hand side of~\eqref{EqSimpleLaplace} are evaluated at $\xb$ and we use the usual index summation convention. In one dimension, taking into account a non constant density $r$, one has \cite[Chapter 6]{Bender1999}
\begin{multline}\label{EqLaplaceFormula1D} 
  \int_\RR \frac{e^{-u(x)/\epsilon}}{(2\pi\eps)^{1/2}}r(x)dx = \frac{e^{-u/\epsilon}}{(u^{(2)})^{1/2} } \\
   \bigg(r
   + \epsilon \Big( \frac{r^{(2)}}{2u^{(2)}} - \frac{r'
   u^{(3)}}{2 (u^{(2)})^2}  
   - \frac{r \,u^{(4)}}{8 (u^{(2)})^2} + \frac{5 r\, (u^{(3)})^2}{24 (u^{(2)})^3}  \Big) + O(\epsilon^2) \bigg)\,,
  \end{multline}
where again all the quantities involving $u$ and $r$ are evaluated at $\xb$.
Such approximation formulas are ubiquitous in several areas of mathematics and are a very classical subject of interest \cite{WongBook}. It can be found in the literature in different forms, under the name of Laplace method, saddlepoint approximation, or Edgeworth expansion in statistics (see \cite{ReidSaddlepoint,barndorff-nielsen} and \cite{TierneyKadane}). It also appears in statistical physics and particularly in probability in the context of large deviations \cite{bolthausen}. A survey concerned with statistic applications can be found in \cite{ReviewStrawderman}.
Formula \eqref{EqSimpleLaplace} thus quantifies the discrepancy to the Gaussian approximation of such integrals and higher-order expansions are available, see in particular \cite{ShunMcCullagh,Kolassa1997}.

In this article, we are interested in a geometric formulation of the Laplace formula, in its multivariate first-order expansion, in a case where the global minimum is attained on a closed manifold, rather than at a unique point.
Although in Formula \eqref{EqSimpleLaplace}, there is a priori no need to use a particular geometric structure to formulate the results, the first-order Laplace expansion contains fourth-order derivatives of the function $u$ which resemble curvature terms of a metric associated to the Hessian of $u$. The main purpose of our work is to make explicit such a geometric formulation with a metric that only depends on the function $u$. A better geometric understanding of this term is of interest, for instance in recognizing divergence and curvature terms for further downstream applications. Note that there are very few works on a geometric formulation of the terms in the Laplace method; we mention in this direction Amari's work in the context of exponential families \cite{AmariExponential}. A geometric Laplace formula which applies to closed manifolds is presented in \cite{Ludewig} and it is given at any order, however it makes use of an operator which is not explicit in terms of the function $u$.

Let us discuss first a possible issue in developing such a geometric formulation. A standard scheme of proof for the Laplace method consists in using the Morse lemma that finds a local change of coordinates such that the function $u$ becomes a nonnegative quadratic form, thereby trivializing all the higher-order (greater than or equal to $3$) derivatives of $u$ in the Laplace formula. However, in doing so the volume form $r$ has been pushed forward by this diffeomorphism so that it affects the results for instance in Equation \eqref{EqLaplaceFormula1D}, and these quantities are rather implicit in $u$. Due to the presence of such terms, such a reduction does not bring a clear gain for a geometric understanding of the Laplace formula in terms of the function $u$.

What we propose is a geometric study of the first two terms in the asymptotic expansion as $\eps\to 0^+$  of the integral
\[
I(\eps) = \iint_{X\times Y}\frac{e^{-u(x,y)/\eps}}{(2\pi\eps)^{d/2}}\,dr(x,y)\,,
\]
for any given volume form $r$ on $X\times Y$. As it can be expected, $I(\eps)$ concentrates where $u$ is minimal, see for instance~\cite{hwang1980}. We are interested in the particular setting where the set of zeros of $u\ge 0$ is a $d$-dimensional surface in a $2d$-dimensional manifold. This situation naturally arises in optimal transport problems. More precisely, we are concerned with a function $u(x,y)$ where $x,y$ are of equal dimensions, say on a Euclidean space, vanishing on a submanifold that can be described as a graph in these coordinates, i.e. $(x,y(x))$. Then we show that a natural geometry in which the Laplace method can be written is the Kim--McCann geometry \cite{kim2007continuity}. It was proposed as a natural pseudo-Riemannian metric for optimal transport problems since it offered a new interpretation as a curvature tensor of a quantity appearing when studying regularity of transport problems, the so-called Ma--Trudinger--Wang tensor~\cite{ma2005regularity}.

Our main result is, in the context explained above, a Laplace formula at first-order in which all the terms are geometric invariants. The usual case of the Laplace method in Euclidean space, that is for a general function $u$ with a unique non-degenerate point for its minimum, can be retrieved as a particular case of our setting.
Aside from the main contribution, we prove a quantitative version of the standard first-order Laplace method with explicit error bounds. 
Our main result reads

 \begin{theorem*}[Informal]
Let $X,Y$ be two manifolds  of equal dimension $d$.
  Suppose that $u\colon X \times Y \to \R$ is sufficiently smooth, nonnegative and vanishes on a $d$-dimensional manifold in $X \times Y$ that can be described as the graph $\Sigma$ of a diffeomorphism $(x,y(x))$ staying away from $\partial Y$. Then, the following first-order expansion holds:
    \begin{multline*}
      \iint_{X\times Y}\frac{e^{-u(x,y)/\eps}}{(2\pi\eps)^{d/2}}f(x,y)\,d\mt(x,y) = \int_{\Sigma} fdm \,+ \\ \eps\int_\Sigma\Big[- \frac18 \Deltat f +\frac14 \nablat_H f + f \Big(\frac{3}{32}{\Rt}  - \frac{1}{8}R + \frac{1}{24}\bracket{h,h} -\frac{1}{8}\bracket{H,H} \Big)\Big] \,dm \\
      + \eps\int_{\partial\Sigma} \frac14 \bracket{\nabla f - K\nabla^N\!f + fKH, \nu}\,d\sigma  + O(\eps^2) ,
    \end{multline*}
    with a term $O(\eps^2)$ controlled explicitly.
\end{theorem*}

In the formula above, there are two (pseudo-) Riemannian metrics involved, $\gt$ and $g$. On $X\times Y$, $\gt$ is the (pseudo-Riemannian) metric introduced by Kim and McCann in \cite{kim2007continuity}, henceforth called the Kim--McCann metric. Once restricted to $\Sigma$, it gives, under additional conditions, a Riemannian metric $g$. The volume forms $\mt, m$ are, up to rescaling, the ones induced by $\gt$ and $g$; $\Rt$ and $R$ are the respective scalar curvatures of $\gt$ and $g$, and
$\Deltat$ denotes the pseudo-Riemannian Laplacian associated to $\gt$. 
Associated to $\Sigma$ seen as a submanifold of $X\times Y$, the second fundamental form is denoted by $h$ and $H$ is the mean curvature. Then, $\nablat_{\!H}$ is the covariant derivative in direction $H$. The quantities $\bracket{h,h}$ and $\bracket{H,H}$ denote the pseudo-norms of the second fundamental form and the mean curvature, respectively. Finally, $K$ is the para-complex structure coming from the Kim--McCann geometry and $\nabla^N\!f$ is the normal component of the gradient of $f$.

The material involving the Kim--McCann geometry is derived in Section \ref{sec:KimMcCann} in which we detail the inner and outer geometry of the submanifold $\Sigma$ and derive useful geometric quantities, which are of interest in themselves. The quantitative estimates for the Laplace formula are detailed in Section \ref{SecQuantitativeLaplace} and the main result is given in Section~\ref{sec:geometric_laplace}.

Our proposed framework of a function $u$ on a product manifold $X \times Y$ might a priori seem too constrained to encompass the usual Laplace formula, for instance the one-dimensional case with a unique nondegenerate minimizer. In particular, if any curvature terms had to be expected in the Laplace formula \eqref{EqLaplaceFormula1D}, this would likely be a quantity similar to the curvature of the corresponding graph of some function evaluated at the critical point. Based on the Kim--McCann metric, we propose in Section \ref{sec:applications} a possible solution for a geometric formulation of the standard multidimensional Laplace method that also applies in the one dimensional case with a nondegenerate global minimum. 
Furthermore, this framework of a decomposition into a product space also naturally appears in different situations, a natural one being the parametrix of the heat kernel treated in Section \ref{SecHeatKernel}.
More generally, our result allows to sometimes write simple formulas for the Laplace method, for instance in the case of the likelihood in Bayesian modelling. 
In optimal transport, this decomposition is readily present and the entropic regularization method leads to such integrals. 
A directly related application, which will be treated in a separate article, is the Taylor expansion of the entropic potentials with respect to the regularization parameter.

\paragraph{Notation.} Coordinates on $X$ are denoted by $x^i, x^j,\dots$, while coordinates on $Y$ are denoted by $y^\ib,y^\jb,\dots$ with barred indices. 
We write partial derivatives as $\partial_i=\frac{\partial}{\partial x^i}$, $\partial_\ib=\frac{\partial}{\partial y^\ib}$, $\partial_{ij}=\frac{\partial^2}{\partial x^i\partial x^j}$, $\partial_{i\jb}=\frac{\partial^2}{\partial x^i\partial y^\jb}$, etc. For the derivatives of $c$ and $u$ we write 
\[
    c_i\coloneqq \partial_ic, \: c_{ij}\coloneqq\partial_{ij}c,\: u_i\coloneqq\partial_iu,
\]
and so on. The $d\times d$ inverse matrix of $c_{i\jb}$ is denoted by $c^{\jb i}$, and we adopt the Einstein summation convention where summation over repeated indices is not explicitly written. We never use $c^{\jb i}$ to raise indices (or $c_{i \jb}$ to lower them). Vector fields on $X\times Y$ are expressed in the coordinate frame $(e_i,e_\ib)$, where we set $e_i=\partial_i$ and $e_\ib=\partial_\ib$.

In general, geometric quantities on $X\times Y$ are denoted with a tilde: $\gt,\mt,\Gammat^k_{ij},\Rt_{i\jb k\lb}$ while quantities without tilde denote objects that live on $\Sigma$: $g,m,\Gamma^k_{ij},R_{ijk\ell}$.

%% file: embeddings.tex
% !TEX root = main.tex
% % % % % % % % % % % % % % % % % % % % % % % % % % % % % % % % 
%                                                             %
%            BACKGROUND ON THE Kim--McCann GEOMETRY            %
%                                                             %
% % % % % % % % % % % % % % % % % % % % % % % % % % % % % % % %

\section{Embeddings in the Kim--McCann geometry}\label{sec:KimMcCann}

Consider the triple $(X, Y,c)$ where $X$, $Y$ are two domains of $\Rd$ and $c(x,y)$, $x\in X, y\in Y$ is a real-valued function.
We assume for simplicity that $X$ and $Y$ are subsets of $\Rd$ but the general idea is that they could be smooth manifolds.
In the context of optimal transport, Kim and McCann~\cite{kim2007continuity} introduced a new pseudo-metric on $X\times Y$ which forms the bedrock of the present work. Their goal was to give a geometric meaning to an intriguing quantity discovered by Ma, Trudinger and Wang~\cite{ma2005regularity} that plays an important role in the regularity of optimal transport. This quantity is now called the MTW tensor and it can be seen as a Kim--McCann curvature tensor. 

Let us first describe the Kim--McCann geometry informally, and delay precise definitions until Section~\ref{sec:review-kim-mccann}. Suppose that we have a bijection $X\to Y$, describing for instance a minimal cost matching between $X$ and $Y$, where the cost of matching an element $x\in X$ to $y\in Y$ is $c(x,y)$. Suppose that we are matching $x\mapsto y$ and $x+\xi\mapsto y+\eta$, where $\xi$ and $\eta$ are small displacements, and we are contemplating whether it would be advantageous to instead match $x\mapsto y+\eta$, $x+\xi\mapsto y$. The (positive or negative) loss we would incur is the \emph{cross-difference}~\cite{McCann_glimpse2014,McCann_line1999}
\begin{equation}\label{eq:cross-difference}
    \delta:= [c(x+\xi,y)+c(x,y+\eta)]-[c(x,y)+c(x+\xi,y+\eta)]\,.
\end{equation}
A Taylor expansion in $\xi,\eta$ gives
\begin{equation*}
    \delta = -D^2_{xy}c(x,y)(\xi,\eta) + o(\abs{\xi}^2+\abs{\eta}^2)\,.
\end{equation*}
The leading-order term $-D^2_{xy}c(x,y)(\xi,\eta)$ is precisely the Kim--McCann metric (see Definition~\ref{def:Kim-McCann} for a proper definition). Note that it only depends on the cost $c$ and in particular does not rely on any Euclidean or Riemannian structure that could exist on $X$ and $Y$.

Frequently in addition to a fixed function $c$ we encounter a more problem-dependent quantity induced by a pair of functions $\phi(x)$ and $\psi(y)$. This quantity, which we denote by $u$, is of the form 
\[
    u(x,y):=c(x,y)-\phi(x)-\psi(y)\,,
\]
and satisfies the properties
\begin{align}
    & \,\,u(x,y)\ge 0, \label{eq:c-div:ineq}\\
    &\inf_y \, u(x,y)=0 \quad\text{for each $x$}, \label{eq:c-div:eq1}\\
    &\inf_x \, u(x,y)=0 \quad\text{for each $y$}.\label{eq:c-div:eq2}
\end{align}
It can be seen as ``rectifying'' the cost $c$ by the addition of $\phi$ and $\psi$ to form a nonnegative quantity. Moreover since $\phi$ only depends on $x$ and $\psi$ only depends on $y$ the function $u$ encodes in some sense the same interaction between $X$ and $Y$ as $c$ did. For example~\eqref{eq:cross-difference} remains unchanged when $c$ is replaced by $u$, and in particular the Kim--McCann metrics induced by $c$ and $u$ are the same. 

In optimal transport~\cite{villani2008optimal} $\phi,\psi$ are the Kantorovich potentials. In matching markets~\cite{Chiappori-McCann-Nesheim-2010,GalichonBook}, $\phi(x)$ and $\psi(y)$ represent the payoffs of (say) worker $x$ and firm $y$ respectively. In information geometry~\cite{AmariBook}, $u$ is called a divergence and generally $X=Y$ and $u$ vanishes on the diagonal $x=y$. We call $u$ a \emph{$c$-divergence} following Pal and Wong~\cite{PalWong_new_information_geometry2018}. 

The $c$-divergence generates a subset of $X\times Y$ defined by 
\[
    \Sigma=\{(x,y) : u(x,y)=0\}\,. 
\]
Assuming that each optimization problem in~\eqref{eq:c-div:eq1},~\eqref{eq:c-div:eq2} is attained at a unique minimizer, $\Sigma$ can then be described as the graph of a map either from $X$ or from $Y$. 

In the next subsections we develop the geometry of $\Sigma$ seen as a submanifold of $X\times Y$; this point of view is at the heart of our Laplace formula. We study the Levi-Civita connection $\nabla$ on $\Sigma$ and the associated Riemann curvature, the second fundamental form and the mean curvature. 

This is in contrast to the earlier approach of Wong and Yang in~\cite{wong2021pseudoriemannian}, in which they establish a link between the Kim--McCann framework and information geometry~\cite{AmariBook}. Wong and Yang showed the importance of the $c$-divergence and developed an ``information'' geometry on $\Sigma$ which is different from the one we present in this paper. Let us explain how. 

When presented with a submanifold $\Sigma\subset X\times Y$, the question arises how to produce a connection on $\Sigma$ from a given connection $\nablat$ on $X\times Y$. Motivated by the product structure of $X\times Y$ and information geometry, Wong and Yang introduce two ``dual connections'' on $\Sigma$ which are defined as the projections of $\nablat$ onto the first ($x$) and second ($y$) components, respectively. They then obtain two curvature tensors, one for each of the two connections. Note that these are not the curvatures induced by the metric $g$ on $\Sigma$. 

Instead we choose to follow the more mainstream route of projecting orthogonally $\nablat$ onto $\Sigma$, or equivalently of studying the Levi-Civita connection of $\Sigma$. This is the standard in submanifold theory~\cite{ONeillBook,chen2014total,dajczer2019submanifold}, and general relativity~\cite{MisnerThorneWheelerBook} for instance. The advantages of our approach is to 
manipulate common objects (Levi-Civita connection, metric curvature) and it highlights the importance of extrinsic curvature (the second fundamental form). It also gives us advanced tools at our disposal such as the fundamental equations (Gauss and Codazzi equations). Beyond that, for future work it could help connecting our framework to other notions in submanifold theory, such as the first and second variation formulas~\cite{Simons1968,xin2018minimal} and minimal varieties in optimal transport~\cite{Kim-McCann-Warren-calibrates}.
On the flip side, our formulas for the connection, see Prop.~\ref{prop:connections}, and Riemann curvature~\eqref{eq:def-R} are more complicated than Wong and Yang's formulas for the dual connections~\cite[Lemma 3]{wong2021pseudoriemannian} and the dual curvature tensors~\cite[Lemma 6]{wong2021pseudoriemannian}.

Let us conclude this introduction by examples of $c$-divergences.
\begin{example}[The distance squared cost]
    One of the simplest examples of a $c$-divergence is the square of the Euclidean distance,
    \[
    u(x,y) = \frac12 \abs{x-y}^2\,,
    \]
    where $X=Y$ is a Euclidean space. Here $u$ can be seen as coming from either the quadratic cost $c(x,y)=\frac12 \abs{x-y}^2$ with $\phi=\psi=0$ or from the bilinear cost $c(x,y)=-x\cdot y$ with $\phi=\psi=-\frac12\abs{x}^2$; $\Sigma$ is the diagonal $\{(x,x) : x\in X\}$.   
    The Kim--McCann pseudo-metric is $\gt((\xi,\eta),(\xi,\eta)) = \xi\cdot\eta$ and the induced metric on $\Sigma$ is $g(\xi,\xi) = \abs{\xi}^2$.
     
     On a Riemannian manifold, the corresponding cost is the squared Riemannian distance $\frac 12 d(x,y)^2$. Note that this cost is not smooth in general due to the presence of the cut locus. However, it is smooth on a neighborhood of the diagonal if $M$ is compact. The metric on the diagonal $\Sigma$ is the Riemannian metric of $M$, see Section \ref{SecHeatKernel}.
\end{example}

\begin{example}[Bregman divergence] \label{ex:bregmandiv}
    Let $X=Y$ be a $d$-dimensional vector space and let $f$ be a differentiable strictly convex function on $X$. Then 
    \[
    u(x,y) = f(x)-f(y)-\bracket{\nabla f(y), x-y}
    \]
    is called the Bregman divergence of $f$ and we denote it $f(x|y)$. It vanishes on the diagonal $\Sigma=\{(x,x) : x\in X\}$. The Kim--McCann metric is $\gt_{x,y}((\xi,\eta),(\xi,\eta))=\nabla^2f(y)(\xi,\eta)$ and the Riemannian metric on $\Sigma$ is $g_x(\xi,\xi)=\nabla^2f(y(x))(\xi,\xi)$.
\end{example}

\begin{example}[Fenchel--Young gap] \label{ex:fenchelyounggap}
    Let $X$ be a $d$-dimensional vector space and let $Y=X^*$, the dual vector space of $X$. Let $f$ be a differentiable strictly convex function on $X$ and take $c(x,y)=-\bracket{x,y}$, $\phi=-f$ and $\psi=-f^*$, where $f^*$ is the convex conjugate (Legendre--Fenchel transform) of $f$. Consider 
    \[
    u(x,y) = f(x)+f^*(y)-\bracket{x,y}.
    \]
    By the Fenchel--Young inequality $u(x,y)\ge 0$ and $u$ vanishes on $\{(x,\nabla f(x)):x\in X\}$. Moreover it can be checked that 
    \[
    u(x,y) = f\big(x|\nabla f^*(y)\big)\,,
    \]
    amd also $u(x,y) = f^*(y|\nabla f(x))$, so that $u$ is essentially a Bregman divergence up to a reparametrization in one of the two variables.
    Laplace expansions with Fenchel--Young gaps are explored in Section~\ref{sec:fenchel-young-gap}.
\end{example}

\begin{example}[Translation-invariant cost]
    Let $U\colon\Rd\to\R$ be a strictly convex nonnegative function satisfying $U(0)=0$ and consider 
    \[
        u(x,y) = U(x-y).
    \]
    Then $u\ge 0$ and $u$ vanishes on the diagonal $x=y$. Translation-invariant costs are natural in optimal transport~\cite{GangboMcCann1995,santambrogio2015optimal} and in Section~\ref{sec:translation-invariant-cost} we study them to recover the usual Laplace method on $\Rd$. 

    Let us also mention the work of Khan and Zhang~\cite{KhanZhang2020}, in which the authors introduce a natural Kähler geometry associated to translation-invariant costs. This geometry is different from the Kim--McCann geometry and can be seen as a complementary framework.
\end{example}

\begin{example}[Log-divergence]
    Take $X=Y=\{x\in\Rd : x_i> 0\}$ the positive orthant and $\alpha > 0$. Consider the cost $c(x,y) = -\frac{1}{\alpha} \log(1+\alpha \bracket{x,y})$.
    Interestingly, this cost gives rise to a Kim--McCann metric with constant sectional curvature, specifically $-4\alpha$, as shown in \cite[Section 4.1]{wong2021pseudoriemannian}.    
    Let $f\colon X\to\R$ be a differentiable function such that $e^{\alpha f}$ is convex. Then
    \[
    u(x,y) = f(x)-f(y)- \frac{1}{\alpha} \log(1+\alpha \bracket{\nabla f(y),x-y})
    \]
    is a nonnegative function which is $0$ if $x = y$ (here we assume that the quantity inside the logarithm is positive). This ``log-divergence'' was introduced in~\cite{PalWongArbitrage2016,WongLogDiv2018}. Note that when $\alpha \to 0$, we recover the Bregman divergence.
\end{example}

\subsection{Review of the Kim--McCann metric} \label{sec:review-kim-mccann}

Let us recall the main geometric objects introduced by Kim and McCann in~\cite{kim2007continuity}.

\begin{definition}\label{def:Kim-McCann}
    The Kim--McCann metric is the pseudo-metric defined on the product space $X\times Y$ by 
    \[
        \gt(x,y) = - \frac 12 \begin{pmatrix}0 & D^2_{xy}c(x,y) \\
            D^2_{xy}c(x,y) & 0\end{pmatrix}.
    \] 
    In other words, for a vector field $(\xi(x,y),\eta(x,y))$ on $X\times Y$ we have $\gt(x,y)((\xi,\eta),(\xi,\eta)) = -D^2_{xy}c(x,y)(\xi,\eta)$, where the mixed partial derivatives part of the Hessian $D^2_{xy}c(x,y)(-,-)$ is understood as a bilinear form.
\end{definition}

Due to its particular form, the signature of this pseudo-metric is $(d,d)$ if it is non-degenerate. Indeed, the \emph{para-complex} structure $T(X\times Y)\to T(X\times Y)$ defined by 
\begin{equation}\label{eq:def-K}
    K(\xi,\eta) = (\xi,-\eta)
\end{equation}
is an isomorphism of eigenspaces with eigenvalues $1$ and $-1$. The space
$X\times Y$ endowed with $(\gt,K)$ is known as a para-Kähler manifold~\cite{parakahlerReview,parakahler}. Since $K^2$ is the identity we have that $K=K^{-1}$. Additionally $K$ is skew-symmetric with respect to $\gt$ in the sense that 
\begin{equation} \label{eq:K-skew}
    \bracket{KU,V} = -\bracket{KV,U},
\end{equation}
where we write $\bracket{-,-}=\gt(-,-)$. This also implies that $\bracket{KU,KV} = -\bracket{U,V}$.

The volume form on $X\times Y$ induced by $\gt$ is $\sqrt{\abs{\det \gt}} = 2^{-d} \abs{\det D^2_{xy}c}$. To avoid writing the factor $2^{-d}$ everywhere we instead define 
\begin{equation} \label{eq:def-mt}
    \mt = \abs{\det D^2_{xy}c}.
\end{equation}
Due to the particular structure of $\gt$ several terms in the Christoffel symbols and the curvature tensor vanish. Indeed the only nonzero Christoffel symbols are
\begin{equation}\label{EqChristoffelsKimMcCann}
    \Gammah^k_{ij} = c^{k\mb}c_{\mb ij}, \quad \Gammah^\kb_{\ib\jb} = c^{\kb m}c_{m\ib\jb}\,,
\end{equation}
and all the other combinations of barred and unbarred indices vanish.
The Riemann curvature tensor is defined as 
\begin{equation} \label{eq:def-Rt}
    \Rt(U, V)W = \nablat_V \nablat_U W - \nablat_U\nablat_V W  - \nablat_{[V, U]} W\,,
\end{equation}
following the sign convention of~\cite{kim2007continuity}, which is also the one of~\cite{ONeillBook}. The only components of the curvature tensor that do not vanish are those for which the number of barred and unbarred indices is equal and one has, for these terms
\begin{equation} \label{EqCurvatureTensorRelations}
    \begin{aligned}
    & \Rt_{ij\kb\lb} = 0\,,\\
    & \Rt_{i\jb \kb \ell} = \frac12 \big(c_{i\jb \kb \ell} - c_{i\ell \bar s }c^{\bar s t}c_{\jb\kb t}\big)\,.
    \end{aligned}
\end{equation}
The other components follow by the standard symmetries of the curvature tensor. In particular we note that 
\begin{equation*}
    \Rt_{i\jb\kb\ell} = \Rt_{i\kb\jb\ell}\,,
\end{equation*}
which directly follows from~\eqref{EqCurvatureTensorRelations} or alternatively from the first Bianchi identity. As a direct consequence the Ricci curvature is 
\[
    \Rt_{i\kb} = -2c^{\jb\ell}\Rt_{i\jb\kb\ell}\,,
\]
and vanishes when the number of barred and unbarred indices is not equal. The scalar curvature is then
\begin{equation}\label{EqRicciHat}
\hR = 8 c^{i \jb}c^{\kb\ell} \hR_{i\jb\kb \ell}\,.
\end{equation}

\subsection{Calculus on \texorpdfstring{$\Sigma$}{Sigma}}

Let us now suppose that we have in addition to $c$ two functions $\phi(x)$ and $\psi(y)$ and that the function $u(x,y)=c(x,y)-\phi(x)-\psi(y)$ satisfies~\eqref{eq:c-div:ineq}--\eqref{eq:c-div:eq2}. In optimal transport language this says that $\phi$ and $\psi$ are $c$-conjugate. We assume that $\Sigma=\{(x,y) : u(x,y)=0\}$ is a smooth submanifold of $X\times Y$ that is the graph of a smooth map $X\to Y, x\mapsto y(x)$ as well as the graph of the inverse map $y(X)\to X, y'\mapsto x(y')$. 

Denote by $t_i$ the pushforward of the vector $e_i=\partial_i$ on $X$ by the embedding $\iota\colon x \mapsto (x,y(x))$, which gives 
\[
    t_i = e_i + \partial_iy^\ib e_\ib.
\]
Note that $(t_i)$ is a basis of the tangent bundle $T\Sigma$. We then define $n_i = K(t_i)$, i.e. 
\[
    n_i =  e_i - \partial_iy^\ib e_\ib.
\]
Then $(n_i)$ gives a basis of the normal bundle $T^{\perp}\Sigma$ (this can be checked directly or using that $K$ is an involution). The inverse formulas read $e_i=\frac{t_i+n_i}{2}$, $e_\ib=\frac{\partial x^i}{\partial y^\ib}\frac{t_i-n_i}{2}$, so that a vector field $U$ on $X\times Y$ can be decomposed on $\Sigma$ into tangent and normal components as
\begin{equation}\label{eq:ei-eib-ti-ni}
    U^ie_i+U^\ib e_\ib = \frac 12 \Big(U^i +\frac{\partial x^i}{\partial y^\ib} U^\ib\Big)t_i + \frac 12 \Big(U^i -\frac{\partial x^i}{\partial y^\ib} U^\ib\Big) n_i\,.
\end{equation}

We compute with coordinates on $\Sigma$ using the embedding $\iota$ define above. Therefore tangent vector fields on $\Sigma$ are expressed in the frame $t_i\coloneqq d\iota(e_i)$, as usual with embeddings. However we also have to deal with more complicated quantities such as normal vector fields, valued in the normal bundle $T^\perp\Sigma$. To streamline computations we adopt the following coordinate representation of these objects.

\begin{notat}[Coordinate representation of normal field] \label{notation:normal}
    
    Whenever $N$ is a normal vector field, we define
    \[
        N'=K(N),
    \]
    where $K$ is defined by~\eqref{eq:def-K}. Since $K$ maps $T^\perp\Sigma$ to $T\Sigma$, this turns $N$ into a tangent vector field $N'$. Then, we express $N'$ in coordinates, $N'=N'^k t_k$ and we systematically drop the prime and always write 
    \[
        N'=N^k t_k.
    \]
    Note that since $K=K^{-1}$ and $n_k=Kt_k$ we have $N=N^kn_k$. The reason we prefer $N'$ when computing with tensors in coordinates is that we can use classical formulas for the covariant derivative on $\Sigma$ (see for instance Section~\ref{sec:second-fundamental-form}).

    More general tensor fields $T$ that involve the normal bundle are expressed in coordinates in the same way: using $K$ when necessary we define a tensor $T'$ which only acts on $T\Sigma$ (and the cotangent space) and then express $T'$ in coordinates. For instance for the second fundamental form $h$ we define $h'(U,V)=Kh(U,V)$ and then write $h'(t_i,t_j)=:h_{ij}^kt_k$. 
\end{notat}

We now present a set of ``computational rules'' valid on $\Sigma$. These are
\begin{align}
    U^\ib &=\partial_iy^\ib U^i\quad\text{when $U$ is tangent},\label{eq:calculus:rule1}\\
    u_{ij} &= -c_{i\jb} \partial_jy^\jb\,, \label{eq:calculus:rule2}\\
    g_{ij} &= u_{ij}\,. \label{eq:calculus:rule3}
\end{align}
The first rule~\eqref{eq:calculus:rule1} says that when $U=U^ie_i+U^\ib e_\ib$ is a vector field on $X\times Y$ that is tangent to $\Sigma$, the object $\partial_iy^\ib$ can be used to change an unbarred index to a barred index. 
\noindent
The second rule~\eqref{eq:calculus:rule2} relates the three natural tensors of order $2$ on $\Sigma$. Because $u\ge 0$, the derivative $u_i$ identically vanishes on $\Sigma$. Differentiating in the direction $t_j$ (tangent to $\Sigma$) we obtain $0=\nablat_{t_j}u_i=u_{ij} + \partial_jy^\jb u_{i\jb}$, and note that $u_{i\jb}=c_{i\jb}$ since mixed derivatives of $u$ and $c$ are always equal. This shows~\eqref{eq:calculus:rule2}.
\noindent
The third rule~\eqref{eq:calculus:rule3} expresses the metric $g$ induced by $\gt$ on $\Sigma$ in terms of $u$. If $U=U^ie_i+U^\ib e_\ib,V=V^je_j+V^\jb e_\jb$ are two tangent vector fields then
\begin{align*}
     \gt(U,V) &= U^iV^\jb \gt(e_i,e_\jb) + U^\ib V^j \gt(e_\ib,e_j) \\
     &= U^i \partial_jy^\jb V^j (-\frac12 c_{i\jb}) + \partial_iy^\ib V^j (-\frac12 c_{\ib j}) \\
     &= u_{ij} U^iV^j\,
\end{align*}
where the last inequality follows from~\eqref{eq:calculus:rule2} and the symmetry of $u_{ij}$. Therefore $u_{ij}(x,y(x)) = g_{ij}(x)$.

We note that $g$ is a priori not necessarily definite. However, in the rest of the paper, we assume that it is the case. This condition is the non-degeneracy condition in \cite{kim2007continuity}.

\subsection{Second fundamental form and projected connections} \label{sec:second-fundamental-form}

Viewing $\Sigma$ as a submanifold of $X\times Y$ leads to natural geometric objects: the second fundamental form which measures the extrinsic curvature of $\Sigma$ embedded in $X\times Y$, and connections on the tangent bundle $T\Sigma$ as well as the normal bundle $T^\perp\Sigma$.

We recall some basic submanifold theory and point to~\cite{ONeillBook} for a reference on the subject. Let $\nablat$ denote the Levi-Civita connection on $(X\times Y,\gt)$. Let $U$, $V$, $N$ be vector fields on $X\times Y$ such that $U$ and $V$ are tangent to $\Sigma$ and $N$ is normal to $\Sigma$. Then on $\Sigma$ the covariant derivative $\nablat_UV$ can be decomposed into
\begin{equation}\label{EqSecondFundamentalFormFirst}
    \nablat_UV = \nabla_UV + h(U,V),
\end{equation}
where $\nabla_UV$ and $h(U,V)$ denote the orthogonal projections of $\nablat_UV$ onto the tangent bundle $T\Sigma$ and the normal bundle $T^\perp\Sigma$, respectively. Similarly we can decompose
\[
    \nablat_U N = -A_N(U) + \nabla^\perp_U N,
\]
where the shape operator $A$ is valued in $T\Sigma$ and $\nabla^\perp$ is a torsion-free connection on the normal bundle.

\begin{definition}[$h$ and $H$]
    As defined by Formula~\eqref{EqSecondFundamentalFormFirst}, $h$ is called the \emph{second fundamental form}. The \emph{mean curvature} $H$ is a normal vector field defined as the trace of $h$ with respect to $g$.
\end{definition}

We list important known results, relevant for the rest of the paper.
\begin{prop}[\cite{ONeillBook}]
    The tangent part $\nabla_UV$ of $\nablat_UV$ is precisely the Levi-Civita connection of the induced metric $g$ on $\Sigma$. 
    
    The second fundamental form is a symmetric bilinear form, which is valued in the normal bundle $T^\perp\Sigma$.

    The second fundamental form and the shape operator are manifestations of the same object since they satisfy the identity $\bracket{A_N(U),V} = \bracket{h(U,V),N}$, where $\bracket{-,-}=\gt(-,-)$. 
\end{prop}

Let us now derive the expressions of $\nabla_UV$, $h(U,V)$ and $\nabla^\perp_U N$ in our framework.

First we recall that our coordinate representations always refer to objects valued in the tangent bundle, as described in Notation~\ref{notation:normal}. Notably $H^kt_k=H'$ with $H'=KH$ and this implies that $H=H^kn_k$. Similarly $h_{ij}^kt_k=h'(t_i,t_j)$ with $h'=Kh$, so that 
\[
    h(t_i,t_j)=h_{ij}^kn_k\,.
\]
In addition, we work more favorably with the purely covariant version 
\begin{equation*}
    h'(U,V,W)\coloneqq\bracket{Kh(U,V),W}\,,
\end{equation*}
where $W$ is tangent and $\bracket{-,-}=\gt(-,-)$. In that way the second fundamental form can be seen as a (scalar) trilinear form on $T\Sigma$. In coordinates we have 
\[
    h'_{ijk}=h'(t_i,t_j,t_k)=\bracket{h^\ell(t_i,t_j) t_\ell,t_k} = h_{ij}^\ell g_{k\ell}.
\]
Thus we see that the index is lowered using the metric $g$ as usual, and from now on we systematically drop the prime and write $h_{ijk}= h_{ij}^\ell g_{k\ell}$. 
Importantly we don't need to worry about the placement of indices because of the following result, a direct consequence of ~\eqref{eq:formula-uijk} in Lemma~\ref{lemma:derivatives-c-divergence}.

\begin{prop}\label{prop:symmetryofh}
    $h'(U,V,W)$ is totally symmetric in $U, V, W$.
\end{prop} 

Translated into the language of information geometry thanks to Wong and Yang's article \cite{wong2021pseudoriemannian} (see also the discussion at the beginning of Section~\ref{sec:KimMcCann}), $h'$ seems to be related to the so-called \emph{cubic tensor}, defined as the difference of the dual connections and which is known to be symmetric~\cite{AmariBook}. 
We note that the cubic tensor seems to remain a bit of a mysterious quantity. 
Thus connecting it to the second fundamental form may be of interest to the information geometry community.

We also record that $\gt(t_i,t_j)=g_{ij}$ (this is by definition) and $\gt(n_i,n_j)=-g_{ij}$ (this follows from~\eqref{eq:K-skew}). Thus in the basis $((t_i)_i,(n_i)_i)$ the Kim--McCann metric takes the form 
\begin{equation*}
    \begin{pmatrix} g_{ij} & 0\\
        0& -g_{ij}\end{pmatrix}.
\end{equation*}
We now ready to state the main results of this section.

\begin{prop}\label{prop:connections}
    The second fundamental form is given by
    \begin{equation}\label{EqFormulaForh}
        h_{ij}^k=\frac 12 \Big(\Gammat_{ij}^k - \frac{\partial x^k}{\partial y^\kb}\frac{\partial y^\ib}{\partial x^i}\frac{\partial y^\jb}{\partial x^j} \Gammat_{\ib\jb}^\kb - \frac{\partial x^k}{\partial y^\kb}\frac{\partial^2 y^\kb}{\partial x^i\partial x^j}\Big)\,.
    \end{equation}
    The mean curvature is 
    \begin{equation} \label{eq:def-H}
        H^k=u^{ij}h^k_{ij}.
    \end{equation}
    The Christoffel symbols for the Levi-Civita connection $\nabla$ on $\Sigma$ are 
    \begin{equation}\label{EqConnection1}
        \Gamma_{ij}^k=\frac 12 \Big(\Gammat_{ij}^k + \frac{\partial x^k}{\partial y^\kb}\frac{\partial y^\ib}{\partial x^i}\frac{\partial y^\jb}{\partial x^j} \Gammat_{\ib\jb}^\kb + \frac{\partial x^k}{\partial y^\kb}\frac{\partial^2 y^\kb}{\partial x^i\partial x^j}\Big)\,.
    \end{equation}
    
\end{prop}

\begin{prop} \label{prop:K}
    The involution $K$ is parallel with respect to $\nablat$, or in other words $K$ and $\nablat$ commute in the sense that for any vector fields $U$ and $V$, $K(\nablat_U V)=\nablat_UK(V)$. 
    
    In particular the normal connection can be obtained from the tangent connection (Levi-Civita on $\Sigma$): if $U$ is tangent, $N$ is normal and $V=K(N)$ then
    \begin{equation}\label{eq:K-commutes}
        \nabla_U^\perp N = K(\nabla_UV)\,.
    \end{equation}
\end{prop}

Before proving Prop.~\ref{prop:connections} and~\ref{prop:K}, let us explain how we will use~\eqref{EqFormulaForh} and~\eqref{eq:K-commutes}. Thanks to~\eqref{eq:K-commutes} we can differentiate normal vector fields as if they were tangent vector fields. In particular because we defined $n_j=K(t_j)$ we have 
\[
    \nabla^\perp_{t_i} n_j=\nabla_{t_i}(Kt_j)=K(\nabla_{t_i}t_j)=K(\Gamma^k_{ij}t_k)=\Gamma^k_{ij}n_k\,.
\]
In other words, the Christoffel symbols for $\nabla^\perp$ and $\nabla$ are the same. 
This explains our choice of coordinate representation: we can do Ricci calculus as usual and write formulas such as 
\[
    \nabla_iH^j = \partial_iH^j + \Gamma^j_{ik} H^k
\]
or
\[
    \nabla_\ell h_{ijk} = \partial_\ell h_{ijk}-(\Gamma^s_{k\ell} h_{ijs}+\Gamma^s_{j\ell} h_{isk}+\Gamma^s_{i\ell} h_{sjk})\,.
\]
As for~\eqref{EqFormulaForh}, it allows us to express second derivatives of $y(x)$ in terms of $h$:
\begin{equation}\label{eq:ddy_h}
    \partial_{ij}y^\kb = 2c^{\kb m} h_{i j m} -c^{\kb m} c_{m\ib \jb} \partial_{i} y^\jb \partial_{j} y^\ib+ c^{k \mb} c_{i j \mb} \partial_{k} y^\kb\,,
\end{equation}
where we prefer to work with the $(0,3)$ version $h_{ijm}$. Summing~\eqref{EqFormulaForh} and~\eqref{EqConnection1} and we also obtain
\begin{equation} \label{eq:Gamma-Gammat-h}
    \Gamma^k_{ij} = \Gammat^k_{ij}-h^k_{ij} = c^{k\mb} c_{ij\mb} - h^k_{ij}\,.
\end{equation}

\begin{proof}[Proof of Prop.~\ref{prop:connections}]
    Let $U$ and $V$ be two vector fields on $X\times Y$ that are tangent to $\Sigma$. We have 
    \begin{align*}
    \nablat_U V&=U^i\nablat_{e_i}( V^je_j+ V^\jb e_\jb)+U^\ib\nablat_{e_\ib}( V^je_j+ V^\jb e_\jb)\\
    &= U^i\partial_i V^je_j +U^i V^j\hGamma_{ij}^ke_k + U^i\partial_i V^\jb e_\jb  \\ 
    &\qquad + U^\ib\partial_\ib V^je_j+U^\ib\partial_\ib V^\jb e_\jb
    +U^\ib V^\jb\hGamma_{\jb\ib}^\kbar e_\kbar\,.
\end{align*}
Since $U$ and $ V$ are tangent to $\Sigma$ we have $U^\ib=\partial_iy^\ib U^i$ and $ V^\jb=\partial_jy^\jb V^j$ by~\eqref{eq:calculus:rule1}. Thus 
\[
    U^i\partial_i V^j+U^\ib\partial_\ib V^j=U^i\partial_i\{ V^j(x,y(x))\}\,.
\]
Moreover using~\eqref{eq:calculus:rule1} twice we write 
\begin{align*}
    U^i\partial_i V^\jb+U^\ib\partial_\ib V^\jb &= U^i\partial_i\{ V^\jb(x,y(x))\}= U^i\partial_i\{ \partial_jy^\jb(x)V^\jb(x,y(x))\}\\
    &= \frac{\partial y^\jb}{\partial x^j}U^i\partial_i\{ V^j(x,y(x))\}+U^i V^j \frac{\partial^2 y^\jb}{\partial x^i\partial x^j}\,.
\end{align*}
Grouping terms we deduce that 
\begin{equation*}
    \nablat_U V = U^i\partial_i\{ V^j(x,y(x))\} t_j + U^i V^j \frac{\partial^2 y^{\bar k}}{\partial x^i\partial x^j}e_{\bar k} + U^i V^j\hGamma_{ij}^ke_k  + U^\ib V^\jb\hGamma_{\ib\jb}^\kbar e_\kbar\,.
\end{equation*}
By using~\eqref{eq:ei-eib-ti-ni} we can express each term in the frame $(t_i,n_j)$ and match against the desired expression
\[
    \nablat_U V = U^i\partial_i\{ V^j(x,y(x))\} t_j + U^iV^j\Gamma^k_{ij}t_k + U^iV^jh^k_{ij}n_k\,.
\]
This gives~\eqref{EqFormulaForh} and~\eqref{EqConnection1}. As for~\eqref{eq:def-H} it directly follows from~\eqref{eq:calculus:rule3}. 
\end{proof}

\begin{proof}[Proof of Prop.~\ref{prop:K}]
    The commuting property $K(\nablat_UV)=\nablat_UK(V)$ can be checked directly. Take $U=e_i$ and $V=ve_j$ for a scalar function $v$. Then $KV=V$ and 
    \[
        \nablat_UV = \partial_iv^je_j + v\Gammat^k_{ij}e_k\,.
    \]
    Therefore $K(\nablat_UV)=\nablat_UV$. 
    
    When $V=ve_\jb$, we have $KV=-V$ and 
    \[
        \nablat_UV = \partial_iv^je_\jb + 0\,.
    \]
    Therefore $K(\nablat_UV)=-\nablat_UV$. A similar argument works for $U=e_\ib$. 

    Formula~\eqref{eq:K-commutes} then follows since $\nabla^\perp_U N$ is the normal component of $\nablat_UN=K(\nablat_UV)=K(\nabla_UV + h(U,V))$ whose normal component is $K(\nabla_UV)$ since $K$ maps $T\Sigma$ to $T^\perp\Sigma$ and vice versa.
\end{proof}

\subsection{Curvatures: the Gauss equation}

The Gauss equation relates several intrinsic and extrinsic curvatures of the embedded manifold $\Sigma$. First let us define the curvature tensor intrinsic to $\Sigma$,
\begin{equation} \label{eq:def-R}
    R(U,V)W = \nabla_V \nabla_UW - \nabla_U\nabla_VW - \nabla_{[V,U]}W\,.
\end{equation}
Note that it follows the same sign convention as~\eqref{eq:def-Rt}. The Gauss equation is~\cite{ONeillBook}
\begin{equation} \label{eq:gauss-general}
    \bracket{\Rt(U,V)W,Z}=\bracket{R(U,V)W,Z} + \bracket{h(U,Z), h(V,W)} - \bracket{h(V,Z), h(U,W)} \,,
\end{equation}
where $U,V,W,Z$ are any tangent vectors and $\bracket{-,-} = \gt(-,-)$. In coordinates we obtain the following result.

\begin{lemma}[Gauss equation] \label{lemma:Gauss-equation}
    We have
    \begin{multline} \label{eq:formula-gauss-other}
        \Rt_{i \jb \kb \ell} \partial_j y^\jb \partial_k y^\kb+\Rt_{i \jb k \lb} \partial_j y^\jb \partial_\ell y^\lb + \Rt_{\ib j \kb \ell} \partial_i y^\ib \partial_k y^\kb+\Rt_{\ib j k \lb} \partial_i y^\ib \partial_\ell y^\lb \\
        = R_{i j k \ell} + h_{i k s} h_{j \ell t} u^{s t} - h_{i \ell s} h_{j k t} u^{s t}\,.
    \end{multline}
    Contracting twice leads to the formula    
    \begin{equation} \label{eq:formula-gauss}
        \Rt_{i \jb k \lb} u^{i k} u^{\jb \lb} = - \frac18 \Rt + \frac12 R - \frac12 \bracket{H,H} + \frac12 \bracket{h,h}\,.
    \end{equation}
\end{lemma}

\begin{proof}
    In formula~\eqref{eq:gauss-general} choose $U=t_i,V=t_j,W=t_k,Z=t_\ell$. Then writing $t_i=e_i+\partial_iy^\ib e_\ib$ and similarly for $t_j$, $t_k$ and $t_\ell$ we have 
    \[
        \bracket{\Rt(t_i,t_j,t_k),t_\ell} = \Rt_{i \jb \kb \ell} \partial_j y^\jb \partial_k y^\kb+\Rt_{i \jb k \lb} \partial_j y^\jb \partial_\ell y^\lb + \Rt_{\ib j \kb \ell} \partial_i y^\ib \partial_k y^\kb+\Rt_{\ib j k \lb} \partial_i y^\ib \partial_\ell y^\lb\,.
    \]
    As for the right-hand side of~\eqref{eq:gauss-general}, we have 
    \begin{equation*}
        \bracket{h(U,Z), h(V,W)} = \bracket{h_{i\ell}^sn_s, h_{jk}^tn_t} = h_{i\ell}^sh_{jk}^t \bracket{n_s,n_t} = - h_{i\ell}^sh_{jk}^t u_{st}\,.
    \end{equation*}
    Lowering indices and repeating the argument for $\bracket{h(V,Z), h(U,W)}$ leads to~\eqref{eq:formula-gauss-other}.

    Next we perform the contraction by mutliplying by $u^{ik} u^{j\ell}$. In the left-hand side of~\eqref{eq:formula-gauss-other} we obtain 
    \begin{multline*}
        c^{i\kb} c^{\jb\ell} \Rt_{i \jb \kb \ell} 
            + u^{ik} u^{\jb\lb} \Rt_{i \jb k \lb} 
            + u^{\ib\kb}u^{j\ell} \Rt_{\ib j \kb \ell} 
            + c^{\ib k} c^{j\lb} \Rt_{\ib j k \lb}\\ = 2 u^{ik} u^{\jb\lb} \Rt_{i \jb k \lb} + 2c^{i\kb} c^{\jb\ell} \Rt_{i \jb \kb \ell} = 2 u^{ik} u^{\jb\lb} \Rt_{i \jb k \lb} + \frac14 \Rt\,,
    \end{multline*}
    using the symmetries of $\Rt$ and~\eqref{EqRicciHat}. The right-hand side of~\eqref{eq:formula-gauss-other} becomes 
    \[
        R + H_sH_tu^{st} - h_{i \ell s} h_{j k t} u^{s t}u^{ik} u^{j\ell} = R - \bracket{H,H} + \bracket{h,h}.
    \]
    We recall that the minus sign in front of the brackets occurs because $\bracket{n_s,n_t}=-u_{st}$. An alternative point of view is that following the convention outlined in Notation~\ref{notation:normal}, $H_s$ describes here $H'=HK$, and $H_sH_tu^{st}=\bracket{KH,KH}=\bracket{-KKH,H}=-\bracket{H,H}$. Similarly $\bracket{Kh,Kh}=-\bracket{h,h}$.
\end{proof}

\subsection{Various formulas}
We collect below useful formulas for our geometric Laplace expansion.
\begin{lemma}[Laplacian on $X \times Y$]\label{lemma:Laplacians}
    Let $f(x,y)$ be a scalar function. The Laplacian of $f$ with respect to $\gt$ is 
    \begin{equation}\label{EqLaplacianXY}
        \Deltat f = -4c^{\ib j}\partial_{\ib j}f\,.
    \end{equation}
\end{lemma}
\begin{proof}
    The standard formula for the Hessian in coordinates is
    \[
        \nablat^2 f(e_\alpha,e_\beta) = \partial_{\alpha\beta}f - \Gammat_{\alpha\beta}^\gamma\partial_\gamma f\,,
    \]
    where Greek letters $\alpha,\beta,\dots$ denote either barred or unbarred indices. By~\eqref{EqChristoffelsKimMcCann}, 
    \begin{equation*}
        [\nablat^2 f] = \begin{pmatrix}
                \partial_{ij}f-\Gammat^k_{ij}\partial_kf & \partial_{i\jb} f\\
                \partial_{\ib j}f & \partial_{\ib\jb}f-\Gammat^\kb_{\ib\jb}\partial_\kb f
            \end{pmatrix}\,.
    \end{equation*}
    Contracting against the inverse of the Kim--McCann metric $\begin{pmatrix}
        0 & -2c^{i\jb} \\
        -2c^{\ib j} & 0
    \end{pmatrix}$ gives us~\eqref{EqLaplacianXY}.

\end{proof}

\begin{lemma}[Derivatives of $\mt$] \label{lemma:derivatives-mt}
    On $X\times Y$ we have the formulas
    \begin{equation} \label{eq:formula-dmt}
        \partial_i \mt = c^{j\kb} c_{ij\kb}\,\mt\,,
    \end{equation}
    and
    \begin{equation}\label{eq:formula-ddmt}
        \partial_{ij}\mt = \big(c^{k \lb} c_{i j k \lb} 
        - c^{k \nb} c^{\lb m} c_{i k\lb } c_{j m \nb} 
        + c^{k\lb} c^{m\nb} c_{ik\lb} c_{jm\nb}\big)\,\mt\,.
    \end{equation}
\end{lemma}
\begin{proof}
    The derivative of the determinant is given by the formula $\partial_\alpha\log\abs{\det \gt_{\beta\gamma}} = \gt^{\beta\gamma}\partial_\alpha\gt_{\beta\gamma}$, which yields the formula giving the derivative of the metric volume form standard in semi-Riemannian geometry 
    \[
        \partial_i\log\mt = \Gammat^j_{ij}\,.
    \]
    Note that here $\mt$ is equal to the volume form up to a multiplicative constant and thus satisfies the same formula. This gives us~\eqref{eq:formula-dmt}. 

    Formula~\eqref{eq:formula-ddmt} follows from~\eqref{eq:formula-dmt} by taking a derivative.
\end{proof}

\begin{lemma}[Derivatives of $u$] \label{lemma:derivatives-c-divergence}    

  On $\Sigma$ we have the formulas
  \begin{equation}\label{eq:formula-uijk}
    u_{ijk} = -2h_{i j k}-(c_{\ib j k} \partial_{i} y^\ib+c_{i \jb k} \partial_{j} y^{\jb}+c_{i j \kb} \partial_{k} y^\kb)\,,
  \end{equation}
  and  
\begin{multline} \label{eq:formula-uijkl}
u_{i j k \ell} = 
-2\partial_\ell {h_{i j k}} - (c_{\ib j k \lb} \partial_i y^\ib 
+ c_{i \jb k \lb} \partial_j {y^\jb} 
+ c_{i j \kb \lb} \partial_k {y^\kb}) \,\partial_\ell {y^\lb}\\
- (c_{\ib j k \ell} \partial_i y^\ib + c_{i \jb k \ell} \partial_j y^\jb + c_{i j \kb \ell} \partial_k y^\kb + c_{i j k \lb} \partial_\ell y^\lb )\\
+ \partial_\ell {y^\lb} c^{\sbar t} (c_{\ib \lb t} c_{j k \sbar} \partial_i y^\ib 
+ c_{i k \sbar} c_{\jb \lb t} \partial_j y^\jb
+ c_{i j \sbar} c_{\kb \lb t} \partial_k y^\kb )\\
+ u^{\sbar \tb}(c_{i j \sbar} c_{k \ell \tb} + c_{i k \sbar} c_{j \ell \tb}+c_{i \ell \sbar} c_{j k \tb}) 
- 2c^{\sbar t}(c_{i j \sbar} h_{k \ell t} + c_{i k \sbar} h_{j \ell t} + c_{j k \sbar} h_{i \ell t}) \,.
\end{multline}

\end{lemma}
\begin{proof}
    When $(x,y)\in\Sigma$ we have the relation 
    \[
        u_{ij}(x,y)=-c_{i\jb}(x,y)\partial_jy^\jb(x)\,,
    \]
    see~\eqref{eq:calculus:rule2}. Differentiating both sides in the direction $t_k$, i.e. applying the operator $\nablat_{t_k}=\partial_k + \partial_ky^\kb\partial_\kb$, we obtain in the left-hand side $u_{ijk} + \partial_ky^\kb u_{ij\kb}$. Since $u$ and $c$ only differ by functions of $x$ only and $y$ only their mixed derivatives always agree, so $u_{ij\kb}=c_{ij\kb}$. 
    In the right-hand side we obtain various derivatives of $c$ and $y(x)$ and we use~\eqref{eq:ddy_h} to substitute second derivatives of the map $y(x)$. This leads to~\eqref{eq:formula-uijk}.
    Doing the same process again, we differentiate~\eqref{eq:formula-uijk} in direction $t_\ell$. In the left-hand side we obtain $u_{ijk\ell} + \partial_\ell y^\lb c_{ijk\lb}$ and in the right-hand side we obtain derivatives of various quantities. We systematically replace second derivatives of $y(x)$ by $h$ quantities thanks to~\eqref{eq:ddy_h}. 
    
    For completeness we also verified formulas~\eqref{eq:formula-uijk} and~\eqref{eq:formula-uijkl} using the symbolic algebra program Cadabra~\cite{peeters2007,cadabra,peeters2007arXiv} which specializes in symbolic tensor computations.\footnote{Code available at \url{https://github.com/flavienleger/geometric-laplace}.}
\end{proof}

Let $f(x,y)$ be a scalar function. The gradient of $f$ with respect to $\gt$ is the vector field $\Gt$ defined by $\gt(\Gt,U)=\nablat_Uf$ for any vector field $U$. On $\Sigma$, we can decompose 
\begin{equation*}
    \Gt=G+N\,,
\end{equation*}
where $G$ and $N$ are tangent and normal vector fields, respectively. Following our convention to only work with coordinates on $T\Sigma$ we then define $N'=KN$. $G$ and $N'$ are expressed in coordinates as 
\[
    G=G^i t_i, \quad N'=N^i t_i.
\]
We also note that $G$ is the gradient of $f$ on $(\Sigma,g)$. Sometimes, when the distinction between vectors and covectors is not so important we write 
\begin{equation}\label{eq:def-G-N}
    \nablat f = \Gt, \quad \nabla f=G\quad\text{and}\quad\nabla^N\!f = N. 
\end{equation}

\begin{lemma}[Derivatives of $f$] \label{lemma:derivatives-f}
    On $\Sigma$ we have the formulas
    \begin{equation} \label{eq:formula-dif}
        \partial_if = \frac12 G_i - \frac12 N_i
    \end{equation}
    and 
    \begin{equation} \label{eq:formula-dijf}
        \partial_{i j} f = -\partial_{\ib j}{f} \partial_{i} y^\ib+\frac{1}{2}\partial_{i} G_{j} - \frac{1}{2}\partial_{i} N_{j}\,.
    \end{equation}
\end{lemma}
\begin{proof}
    For any vector field $U$, $\nablat_Uf = \bracket{G+N,U}$, denoting $\bracket{-,-} = \gt(-,-)$. Taking $U=e_i$, we have 
    \[
        \partial_i f = \nablat_{e_i} f = \bracket{G,e_i} + \bracket{KN', e_i} = \bracket{G,e_i} - \bracket{N',  Ke_i}\,.
    \]
    Note that $Ke_i=e_i$. Also in the basis $(t,n)$ we have $e_i = \frac12 (t_i+n_i)$ and since $G$ and $N'$ are tangent,
    \[
        \bracket{G,e_i} - \bracket{N',  Ke_i} = \bracket{G,\frac12 t_i} - \bracket{N',  \frac12 t_i}\,.
    \]
    We deduce that 
    \[
        \partial_if = \frac12 G^j \bracket{t_j,t_i} - \frac12 N^j \bracket{t_j,  t_i} = \frac12 G^j g_{ij}- \frac12 N^j g_{ij}\,.
    \]
    This proves~\eqref{eq:formula-dif}.

    To obtain~\eqref{eq:formula-dijf}, we keep in mind that in formula~\eqref{eq:formula-dif} the quantity $\partial_if$ is a function of $(x,y)$ while $G_i$ and $N_i$ are functions of $x$ (they are only defined on $\Sigma$ and read through the embedding $x\mapsto y(x)$). Therefore~\eqref{eq:formula-dif} should be understood as 
    \[
        \partial_if(x,y(x)) = \frac12 G_i(x) - \frac12 N_i(x)\,.
    \]
    Differentiating with respect to $x^j$ leads to the desired result, after switching indices $i,j$.
\end{proof}

%% file: geometric_laplace.tex
% !TEX root = main.tex
\section{Geometric Laplace expansion} \label{sec:geometric_laplace}

\subsection{The main result}

Let $X$ and $Y$ be two domains of $\Rd$ and let $u$ be a nonnegative function on the product space $X\times Y$, such that $(X,Y,u)$ satisfies Assumption~\ref{assumption:XYu} below. The Kim--McCann geometry induced by $u$ provides the following structures: a pseudo-Riemannian metric $\gt$ over $X\times Y$ equipped with a special mapping $K$ called a para-complex structure, and a submanifold theory for the vanishing set of $u$. This material is presented in Section~\ref{sec:KimMcCann}.

We note that the Euclidean structure of $X$ and $Y$ inherited from $\Rd$ plays in itself no role in our geometric framework. Thus $X$ and $Y$ could be more general $d$-dimensional smooth manifolds.

\begin{table}[h]
  \centering
    \begin{tabular}{c|c|c} \label{table:geometric-quantities}
      & Quantity & Defined by...\\
      \hline
      On $X\times Y$ & $\gt$ & \eqref{def:Kim-McCann} \\
       & $K$ & \eqref{eq:def-K}\\
       & $\mt$ & \eqref{eq:def-mt}\\
       & $\nablat$ & Levi-Civita connection\\
       & $\Deltat$ & \eqref{EqLaplacianXY}\\
       & $\Rt$ & \eqref{EqRicciHat}\\
       \hline      
      On $\Sigma$ & $R$ & \eqref{eq:def-R} \\
       & $h$ & \eqref{EqSecondFundamentalFormFirst}\\
       & $H$ & \eqref{eq:def-H}\\
       & $\nabla f,\nabla^N\!f$ & \eqref{eq:def-G-N}
    \end{tabular}
  \caption{Geometric quantities}
\end{table}

From the Kim--McCann pseudo-metric can be derived a number of geometric quantities which appear in our Laplace formula. They are defined in Section~\ref{sec:KimMcCann} and listed in Table~\ref{table:geometric-quantities}.

\begin{assumption}[Assumptions on $X,Y,u$] \label{assumption:XYu}
  \leavevmode
\begin{enumerate}[(i)]
    \item \label{assumption:XYu:XY} $X$ and $Y$ are open subsets of $\Rd$ with smooth boundaries or no boundaries and $u$ is a nonnegative measurable function over $X\times Y$. 
    \item\label{assumption:XYu:Sigma} The vanishing set $\Sigma=\{(x,y)\in X\times Y : u(x,y)=0\}$ is the graph $(x,y(x))$ of a map $y\colon X\to Y$ which is a $C^3$-diffeomorphism onto its image.    
    \item\label{assumption:XYu:ball} There exists $\delta>0$ such that $Y$ contains the ball $B(y(x),\delta)$ for all $x\in X$.
    We then define a tubular neighborhood of $\Sigma$,
    \[
      \Sigma_\delta\coloneqq \{(x,y')\in X\times Y : y'\in B(y(x),\delta)\}.
    \] 
    \item\label{assumption:XYu:regularity} $u\in C^6(\Sigma_\delta)$. 
    \item \label{assumption:XYu:lowerbound} There exists $\lambda>0$ such that
    \begin{align*}
      u(x,y') &\ge \frac\lambda 2\abs{y'-y(x)}^2 \quad\text{for all $(x,y')\in \Sigma_\delta,$} \\    
      u(x,y') &\ge \frac\lambda 2\delta^2 \quad\text{for all $(x,y')\in (X\!\times\! Y)\setminus \Sigma_\delta.$}
    \end{align*}
    Here $\abs{\cdot}$ denotes the Euclidean norm in $Y$. 
\end{enumerate}
\end{assumption}

Let us make a few comments on these assumptions. About~\ref{assumption:XYu:XY}, note that the boundaries of $X$ and $\Sigma$ are in a one-to-one correspondance via the map $y(x)$. We ask for $\Sigma$ to have a smooth boundary since the Laplace formula~\eqref{eq:mainthm} contains a boundary term integrated over $\partial\Sigma$. As for $Y$, it doesn't in fact need to have a smooth boundary.

In~\ref{assumption:XYu:Sigma}, we only ask for the map $y(x)$ to be a diffeomorphism onto its image. Indeed we should have the freedom to extend the space $Y$ if we so wish (while keeping $X$ fixed), since in the Laplace method only the neighborhood of the points $y(x)$ really plays a role.

Finally, \ref{assumption:XYu:ball}, \ref{assumption:XYu:regularity} and \ref{assumption:XYu:lowerbound} are roughly the counterparts of Assumption~\ref{ass:ux-localized}\ref{ass:ux-localized:ball}, \ref{ass:ux-localized:regularity} and~\ref{ass:ux-localized:bound-below} respectively.

Before we state our main result, we define the norm 
\[
  \norm{r}_{L^1_xW^{4,\infty}_y(\Sigma_\delta)} = \int_X\int_Y \norm{r(x,\cdot)}_{W^{4,\infty}(B(y(x),\delta))}dx,
\]
where $W^{4,\infty}$ stands for the usual Sobolev space.

\begin{theorem} \label{thm:laplace}
  Suppose that $X$, $Y$ and $u$ satisfy Assumption~\ref{assumption:XYu} and let $r\in L^1_xW^{4,\infty}_y(\Sigma_\delta)\cap L^1(X\times Y)$. Then there exists a constant $C>0$ such that for all $\eps>0$,
  \begin{multline} \label{eq:mainthm}
      \iint_{X\times Y}\frac{e^{-u(x,y)/\eps}}{(2\pi\eps)^{d/2}} \,dr(x,y) = \int_{\Sigma} fdm \,+ \\ \eps\int_\Sigma\Big[- \frac18 \Deltat f +\frac14 \nablat_{\!H} f + f \Big(\frac{3}{32}{\Rt}  - \frac{1}{8}R + \frac{1}{24}\bracket{h,h} -\frac{1}{8}\bracket{H,H} \Big)\Big] \,dm  \\
      + \eps\int_{\partial\Sigma} \frac14 \bracket{\nabla f - K\nabla^N\!f + fKH, \nu}\,d\sigma + \eps^2\mathcal{R}(\eps),
    \end{multline}
    with $f\coloneqq dr/d\mt$ and with 
    \[
      \abs{\calR(\eps)} \le C \big(\norm{r}_{L^1_xW^{4,\infty}_y(\Sigma_\delta)} + \norm{r}_{L^1(X\times Y)}\big)\,.
    \]
    The constant $C$ depends on $\lambda$, $\delta$, $d$ and $\norm{D^ku}_{L^{\infty}(\Sigma_\delta)}$ for $3\le k\le 6$. In the boundary term, $\nu$ is the outer normal and $\sigma$ is the volume form induced by $g$ on $\partial\Sigma$. 
\end{theorem}

In~\eqref{eq:mainthm}, $r$ should be seen as a test function, i.e. a smooth function we integrate against in order to understand $e^{-u(x,y)/\eps}$. Geometrically $r$ is a volume form over $X\!\times\! Y$, which is why we write it as $dr(x,y)$. Then on the right-hand side, $f$ is a scalar function defined as the ratio of two $2d$-forms ($2d$ is the dimension of $X\!\times\! Y$). Observe that $f$ and its derivatives only play a role on $\Sigma$. Therefore we only need to define $f$ on the tubular neighborhood $\Sigma_\delta$. Since $u$ is $C^6$ on $\Sigma_\delta$, the quantity $\mt=\abs{\det D^2_{xy}u}$ is well-defined on $\Sigma_\delta$ and thus so is $f$.

The geometric quantities that appear in~\eqref{eq:mainthm} can be looked up in Table~\ref{table:geometric-quantities}.

The brackets $\bracket{-,-}$ denote the pseudo metric $\gt(-,-)$. 
Since $\gt$ is a non-degenerate bilinear form it extends to tensors of any given type, and can therefore be applied to $h$.

\begin{remark}
Formula \eqref{eq:mainthm} suggests that $u$ not only induces the Kim--McCann metric but also the volume form denoted by $\mt$, which appears in the definition of $f$; these two natural objects being only needed in the neighborhood of $\Sigma$.
In particular, the measure $e^{-u(x,y)/\eps} d\mt(x,y)$ is defined without any reference to the volume form chosen in the Laplace formula and it can be integrated against the function $f$.
\end{remark} 

\paragraph{A convergence of measures.} Writing $r=f\mt$ and viewing $f$ as a scalar test function, Theorem~\ref{thm:laplace} can be interpreted as the measure $\mu_\eps\coloneqq (2\pi\eps)^{-d/2} e^{-u/\eps} \mt$ converging towards a distribution concentrated on $\Sigma$. More precisely, 
\[
  \frac{1}{\eps}\bigg\{\mu_\eps - m\delta_\Sigma - \eps\Big[-\frac 18\Deltat+ \frac 14 \nablat_{\!H}+ \Big(\text{curvatures}\Big)\Big] m\delta_\Sigma\bigg\} \to 0,
\]
as $\eps\to0$, where $\delta_\Sigma$ denotes the Dirac measure supported on $\Sigma$. The above convergence certainly holds in the sense of distribution (i.e. against smooth test functions with compact support). In view of the remainder term $\calR(\eps)$ it also holds in a certain dual Sobolev space that we don't wish to make explicit.

When $f=1$, i.e. $r=\mt$, 
in the right-hand side of~\eqref{eq:mainthm} the $\eps$ term simplifies into a mix of intrinsic and extrinsic curvature terms. These terms are strongly reminiscent of the \emph{second variation formula}, 
which describes how the volume of a family of submanifolds $\Sigma_t$ changes around $\Sigma\coloneqq\Sigma_0$ (see for instance~\cite{Simons1968}). Note that the volume of $\Sigma_t$ is the total mass of the measure $\delta_{\Sigma_t}$ and that the left-hand side of~\eqref{eq:mainthm} is the total mass of $\mu_\eps$ (assuming $r=\mt$). Thus 
our result could be understood as a variation formula around $\Sigma$, but where the variation of $\Sigma$ consists of smoothed out measures $\mu_\eps$ instead of neighboring surfaces $\Sigma_t$.

\subsection{Proof of the main result}

\begin{proof}[Proof of Theorem~\ref{thm:laplace}]
  Write $r(x,y) = f(x,y)\mt(x,y)$ and let
  \[
    I(\eps) = \int_X\int_Y \frac{e^{-u(x,y)/\eps}}{(2\pi\eps)^{d/2}}r(x,y)\,dxdy.
  \]
  In order to obtain a Laplace expansion of $I(\eps)$, we proceed for each $x\in X$ to do the Laplace expansion of $\int_Y \frac{e^{-u(x,y)/\eps}}{(2\pi\eps)^{d/2}}r(x,y)\,dy$, using Corollary~\ref{cor:quantitativelaplace}. 
  
  Here comes a small notational problem: $\int_Y \frac{e^{-u(x,y)/\eps}}{(2\pi\eps)^{d/2}}r(x,y)\,dy$ is an integral over $y$, while Corollary~\ref{cor:quantitativelaplace} uses $x$. We prefer to keep the same variable since that also impacts how we write derivatives, with $y$ derivatives using barred indices and $x$ derivatives using unbarred indices. Therefore for the entirety of this proof, we switch the roles of $X$ and $Y$. Assumption~\ref{assumption:XYu} is adjusted as follows: \ref{assumption:XYu:Sigma} $\Sigma$ is the graph of a function $x(y)$ and \ref{assumption:XYu:lowerbound} $u(x',y)\ge \frac\lambda 2 \abs{x'-x(y)}^2$.

  We therefore freeze $y\in Y$ and 
  do the Laplace expansion of $\int_X \frac{e^{-u(x,y)/\eps}}{(2\pi\eps)^{d/2}}r(x,y)\,dx$. The various assumptions put in place in Assumption~\ref{assumption:XYu} directly correspond to the needed Assumption~\ref{ass:ux-localized} to apply Corollary~\ref{cor:quantitativelaplace}.
  In particular note that $x(y)$ is the unique minimizer of $x'\mapsto u(x',y)$ and it corresponds to $\xb$ in Corollary~\ref{cor:quantitativelaplace}.

  Combining the obtained expansions for each $y$ we have 
  \begin{multline*}
    I(\eps) = \int_Y  \Big(\frac{1}{\sqrt{\det[u_{ij}(x(y),y)]}}\Big[r + \eps\Big(\frac 12 u^{ij}\partial_{ij}r-\frac 12 u_{jk\ell}u^{ij}u^{k\ell}\partial_ir \\
+    \frac 18 ru_{ijk}u_{\ell mn}u^{ij}u^{k\ell}u^{mn}+\frac{1}{12} r u_{ijk}u_{\ell mn} u^{i\ell} u^{jm} u^{kn} \\- \frac 18 r u_{ijk\ell}u^{ij}u^{k\ell}\Big)\Big]_{(x(y),y)}+ \eps^2\,\mathcal{R}(\eps,y)\Big)\,dy,
  \end{multline*}
  where $\calR$ a priori depends on $y$ and satisfies the bound 
  \[
    \abs{\mathcal{R}(\eps,y)}\le  C\,\norm{r(\cdot,y)}_{W^{4,\infty}(B(x(y), \delta))} + \norm{r(\cdot,y)}_{L^1(X)}.
  \]  
  $C$ depends on $d$, $\lambda$, $\delta$ and $\norm{D^k_xu(\cdot,y)}_{L^{\infty}(B(x(y), \delta))}$ for $3\le k\le 6$, which we further bound by $\norm{D^ku}_{L^{\infty}(\Sigma_\delta)}$ which does not depend on $y$. We expand the brackets and break down 
  \[
    I(\eps)=:I_0 + \eps I_1 + \eps^2I_2(\eps).
  \] 
  
  We interpret $I_0 + \eps I_1$ as an integral over $\Sigma$ parametrized by $Y$ via $y\mapsto (x(y),y)$. In order to reveal the volume form $\bar m=\sqrt{\det[u_{\ib\jb}]}$ induced by the Riemannian metric in $y$ coordinates, we take determinants in the identity $u_{ij}=c_{i\jb}c_{\ib j} u^{\ib\jb}$ to obtain 
  \[
    \frac{1}{\sqrt{\det[u_{ij}]}} = \frac{\bar m}{\mt}\,\cdot
  \]
  Therefore 
  \begin{multline*}
    I_0 + \eps I_1 = \int_\Sigma   \Big[\frac{r}{\mt} + \eps\frac{1}{\mt}\Big(\frac 12 u^{ij}\partial_{ij}r-\frac 12 u_{jk\ell}u^{ij}u^{k\ell}\partial_ir + \frac 18 ru_{ijk}u_{\ell mn}u^{ij}u^{k\ell}u^{mn}\\
    +\frac{1}{12} r u_{ijk}u_{\ell mn} u^{i\ell} u^{jm} u^{kn} - \frac 18 r u_{ijk\ell}u^{ij}u^{k\ell}\Big)\Big]\,dm.
  \end{multline*}
  Note that we slightly abused notation by writing $\bar m$ for the ``coordinate expression'' $\bar m(y)dy$ and $m$ for the more abstract geometric volume form on $\Sigma$.

  \paragraph{Term $I_0$.}
  Since $r=f\mt$ we have
  \[
    I_0 = \int_\Sigma f\,dm.
  \]

  \paragraph{Term $I_1$.}
  Let us now simplify 
  \begin{multline*}    
    I_1 = \int_\Sigma  \Big(\frac 12 u^{ij}\frac{\partial_{ij}(f\mt)}{\mt} - \frac 12 u_{jk\ell}u^{ij}u^{k\ell}\frac{\partial_i(f\mt)}{\mt} + \\
    \frac 18 fu_{ijk}u_{\ell mn}u^{ij}u^{k\ell}u^{mn}+\frac{1}{12} f u_{ijk}u_{\ell mn} u^{i\ell} u^{jm} u^{kn} \\ -\frac 18 f u_{ijk\ell}u^{ij}u^{k\ell}\Big)\,dm,
  \end{multline*}
  which we write as 
  \[
    I_1 =: \int_\Sigma L\,dm.
  \]
  We want to express $L$ using only geometric objects (metric, covariant derivative, curvature), either intrinsic and extrinsic to $\Sigma$. We break down $L$ into
  \begin{equation*}
    \begin{aligned}
      L_1 &= \frac12   u^{ij} \frac{\partial_{ij}(f\mt)}{\mt},\\
      L_2 &= -\frac12   u_{jk\ell}u^{ij}u^{k\ell}\frac{\partial_i(f\mt)}{\mt},\\
      L_3 &= \frac{1}{8} f u_{ijk}u_{\ell mn}u^{ij}u^{k\ell}u^{mn} , \\
      L_4 &= \frac{1}{12}fu_{ijk}u_{\ell mn} u^{i\ell} u^{jm} u^{kn}, \\
      L_5 &= -\frac{1}{8} f u_{ijk\ell}u^{ij}u^{k\ell},
    \end{aligned}
  \end{equation*}  
  so that $L = L_1 + L_2 + L_3 + L_4 + L_5$. 
  
  We compute each of these five terms using tedious but straightforward computations. Since some of the formulas involved can become lengthy we have also checked them using the symbolic algebra program Cadabra~\cite{peeters2007,cadabra,peeters2007arXiv} which specializes in symbolic tensor computations.\footnote{Code available at \url{https://github.com/flavienleger/geometric-laplace}.}
  
  Using the product rule where necessary, we replace the quantities $\partial_{ij}f$, $\partial_if$, $\partial_{ij}\mt$, $\partial_i\mt$, $u_{ijk}$, $u_{ijk\ell}$ by their expressions in formulas~\eqref{eq:formula-dijf}, \eqref{eq:formula-dif}, \eqref{eq:formula-ddmt}, \eqref{eq:formula-dmt}, \eqref{eq:formula-uijk} and~\eqref{eq:formula-uijkl}. We obtain after simplification the following expressions.
  \begin{multline*}
    L_1 = \frac14 u^{i j}\partial_{i}{G_{j}} 
    - \frac14 u^{i j}\partial_{i}{N_{j}} 
    + \frac12 c^{i \jb}\partial_{i \jb}{f}
    + \frac12 c^{\ib j} u^{k \ell} c_{\ib j k} G_\ell  
    - \frac12 c^{\ib j} u^{k \ell} c_{\ib j k} N_\ell\\
    + f \Big(\frac12 c^{\ib j} u^{k \ell} c_{\ib j k \ell} 
    + \frac12 c^{i \kb} c^{m \nb} u^{j \ell}  c_{i j \kb} c_{\ell m\nb}
    - \frac12 c^{i \nb} c^{\kb m} u^{j \ell} c_{i j \kb} c_{\ell m \nb}\Big).  
  \end{multline*}
  \begin{multline*}
    L_2 = \Big(- \frac14 c^{\ib \ell} u^{j k} c_{\ib j k} 
    - \frac12 c^{\ib j} u^{k \ell} c_{\ib j k} 
    + \frac12 u^{i j} u^{k \ell} h_{i j k}\Big) (G_\ell - N_\ell)\\
    + f \Big(c^{i \kb} u^{j \ell} u^{m n} c_{i j \kb} h_{\ell m n} 
    - c^{i \kb} c^{m \nb} u^{j \ell} c_{i j \kb} c_{\ell m \nb} 
    - \frac12 c^{i \nb} c^{j \kb} u^{\ell m} c_{i j \kb} c_{\ell m \nb} 
    \Big).
  \end{multline*}
  \begin{multline*}
    L_3 = f \Big(\frac12 c^{i \kb} c^{j \nb} u^{\ell m} c_{i j \kb} c_{\ell m \nb}
    + \frac12 c^{j \kb} c^{m \nb}  u^{i\ell} c_{i j \kb} c_{\ell m \nb}
    + \frac18 u^{i j} u^{\ell m} u^{\kb \nb} c_{i j \kb} c_{\ell m \nb} \\
    - c^{j \kb} u^{i \ell} u^{m n} c_{i j \kb} h_{\ell m n}  
    - \frac12 c^{\kb \ell} u^{i j} u^{m n} c_{i j \kb} h_{\ell m n} 
    + \frac12 u^{i j} u^{k \ell} u^{m n} h_{i j k} h_{\ell m n} \Big).
  \end{multline*}
  \begin{multline*}
    L_4 = f \Big(
    \frac12 c^{j \nb} c^{\kb m} u^{i \ell} c_{i j \kb} c_{\ell m \nb} 
    + \frac14 u^{i \ell} u^{j m} u^{\kb \nb} c_{i j \kb} c_{\ell m \nb} 
    - c^{\kb n} u^{i \ell} u^{j m} c_{i j \kb} h_{\ell m n} 
    + \frac13 u^{i \ell} u^{j m} u^{k n} h_{i j k} h_{\ell m n} \Big).
  \end{multline*}
  \begin{multline*}
    L_5 = f \Big(
    \frac14 u^{i j} u^{k \ell} \partial_i h_{j k \ell}
    + \frac14 c^{i \jb} c^{k \lb} c_{i \jb k \lb} 
    - \frac12 c^{\ib j} u^{k \ell} c_{\ib j k \ell}
    + \frac18  u^{i j} u^{\kb \lb} c_{i j \kb \lb}\\
    - \frac14 c^{i \lb} c^{j \mb} c^{\kb n} c_{i j \kb} c_{\lb \mb n} 
    - \frac18 c^{\kb \ell} u^{i j} u^{\mb \nb} c_{i j \kb} c_{\ell \mb \nb} 
    - \frac18 u^{i j} u^{\kb \lb} u^{m n} c_{i j \kb} c_{\lb m n} 
    - \frac14 u^{i \ell} u^{j m} u^{\kb \nb} c_{i j \kb} c_{\ell m \nb}  \\
    + \frac14 c^{\kb \ell} u^{i j} u^{m n} c_{i j \kb} h_{\ell m n} 
    + \frac12 c^{\ib \ell} u^{j m} u^{k n} c_{\ib j k} h_{\ell m n} 
    \Big).
  \end{multline*}

  Summing $L_1$ through $L_5$, many terms cancel out and after simplification we are left with
  \begin{multline*}
    L = \frac14 u^{i j} \partial_{i}{G_{j}}  
    - \frac14 u^{i j} \partial_{i}{N_{j}} 
    + \frac12  c^{i \jb} \partial_{i \jb} f
    - \frac14 c^{\kb \ell} u^{i j} c_{i j \kb} G_\ell
    + \frac12 G_{i} h_{j k l} u^{i j} u^{k l}
    + \frac14 N_{i} c^{i \jb} u^{k \ell} c_{\jb k \ell} \\
    - \frac12 N_{i} h_{j k l} u^{i j} u^{k l}
    + f \left(\frac14 \partial_{i}{h_{j k l}} u^{i j} u^{k l}
    + \frac14 c^{i \jb} c^{k \lb} c_{i \jb k \lb}
    + \frac{1}{8} u^{i k} u^{\jb \lb} c_{i \jb k \lb} 
    - \frac14 c^{i \ib} c^{j \jb} c^{k {\bar k}} c_{i j {\bar k}} c_{\ib \jb k} \right.\\
    - \frac{1}{8}c^{i \ib} c_{i \jb {\bar k}} c_{\ib j k} u^{j k} u^{\jb {\bar k}} 
    - \frac14 c^{i \ib} c_{\ib j k} h_{i l m} u^{j k} u^{l m} 
    - \frac12 c^{i \ib} c_{\ib j k} h_{i l m} u^{j l} u^{k m}+\frac12 h_{i j k} h_{l m n} u^{i j} u^{k l} u^{m n}\\
    \left.+\frac{1}{3}h_{i j k} h_{l m n} u^{i l} u^{j m} u^{k n}\right).
  \end{multline*}

  We now replace the partial derivatives $\partial_iG_j,\partial_iN_j,\partial_ih_{jk\ell}$ with covariant derivatives. Importantly recall our convention outlined in Notation~\ref{notation:normal}: coordinate expressions always represent quantities in the frame $(t_i)$, thus $G=G^it_i$ since $G$ is a tangent vector field to begin with, $N^it_i = KN$ since $N$ is a normal vector field and $h^k_{ij}t_k = Kh(t_i,t_j)$. We adopt the classical index notation where in an expression such as $\nabla_iG_j$, the covariant derivative is applied first and then the resulting oject is evaluated at $(t_i,t_j)$. We therefore have 
  \[
    \nabla_iG_j = \partial_iG_j - \Gamma^k_{ij} G_k.
  \]
  The Christoffel symbols $\Gamma^k_{ij}$ can be expressed in terms of $c$ and $h$ by~\eqref{eq:Gamma-Gammat-h} which leads to
  \[
    \partial_iG_j = \nabla_iG_j + c^{k\mb} c_{ij\mb} G_k + u^{k\ell} h_{ij\ell} G_k.
  \]
  A similar formula holds for $N$ and for $h$ it takes the form 
  \begin{multline*}
    \partial_{\ell}{h_{i j k}} = \nabla_{\ell} h_{i j k} + c^{\ib m} c_{i \ell \ib} h_{j k m} 
+ c^{\ib m} c_{j \ell \ib} h_{i k m} + c^{\ib m} c_{k \ell \ib} h_{i j m}\\-h_{i j m} h_{k \ell n} u^{m n}
-h_{i k m} h_{j \ell n} u^{m n}-h_{i \ell m} h_{j k n} u^{m n}.
  \end{multline*}
  We also recognize within $L$ the expression of the curvature tensor $\Rt$~\eqref{EqCurvatureTensorRelations}. We obtain after simplification 
  \begin{multline*}
    L = \frac14 \nabla_i G_j u^{i j} - \frac14 \nabla_i N_j u^{i j}+\frac12 c^{i \jb} \partial_{i \jb}{f} +\frac14 G_{i} h_{j k l} u^{i j} u^{k l} - \frac14 N_{i} h_{j k l} u^{i j} u^{k l} - \frac12 c^{i \jb} c^{k \lb} {\tilde R}_{i \jb k \lb} f \\- \frac14 u^{i k} u^{\jb \lb} {\tilde R}_{i \jb k \lb} f+\frac14 \nabla_{i}{h_{j k \ell}} f u^{i j} u^{k l}+\frac14 f h_{i j k} h_{\ell m n} u^{i j} u^{k \ell} u^{m n} - \frac{1}{6}f h_{i j k} h_{\ell m n} u^{i \ell} u^{j m} u^{k n}.
  \end{multline*}
  Here we can recognize several geometric quantities. First recall that the divergence of a tensor is the trace of its covariant derivative. So for a vector field $V=V^it_i$ we have 
  \[
    \div V = \nabla_i V^i = \nabla_i(V_j u^{ij}) = (\nabla_iV_j)u^{ij}.
  \]
  Note that since the metric is compatible with the connection, traces can be taken inside or outside the covariant derivative and there is no ambiguity in writing $\div V = \nabla_iV_ju^{ij}$. 

  In the expression of $L$ we recognize $\nabla_i G_j u^{i j}=\div G$. For $N$ it is the same but recall that $N^it_i = KN$, thus $\nabla_i N_j u^{i j} = \div(KN)$. Finally $\nabla_ih_{jk\ell} u^{ij} u^{k\ell} =\nabla_i(h_{jk\ell} u^{k\ell})u^{ij}  = \nabla_iH_j u^{ij}=\div(KH)$.
  
  We also recognize the $X\times Y$ Laplacian  $c^{i\jb}\partial_{i\jb}f = -\frac14 \Deltat f$ (see Lemma~\ref{lemma:Laplacians}), and the scalar curvature $\hR = 8 c^{\jb l}c^{i\kb} \hR_{i\jb\kb l}$. We obtain 
  \begin{multline*}
    L = \frac14 \div(G) - \frac14  \div(KN) -\frac18 \Deltat f + \frac14 \bracket{G, KH}  - \frac14  \bracket{KN,KH} +\frac{1}{16}f\Rt \\ - \frac14 u^{ik} u^{\jb\lb} \Rt_{i\jb k\lb} f +\frac14 \div(KH) f + \frac14 \bracket{KH,KH} f - \frac16 \bracket{Kh,Kh}.
  \end{multline*}
  We combine $\frac14 \bracket{G, KH} + \frac14 \div(KH) f = \frac14 \div(f KH)$, simplify some $K$'s, and obtain
  \begin{multline*}
    L = \frac14 \div(G) - \frac14  \div(KN) -\frac18 \Deltat f +  \frac14 \div(f KH)  + \frac14  \bracket{N,H} +\frac{1}{16}f\Rt \\ - \frac14 u^{ik} u^{\jb\lb} \Rt_{i\jb k\lb} f - \frac14 \bracket{H,H} f + \frac16 \bracket{h,h}.
  \end{multline*}
  Finally for the quantity $u^{ik} u^{\jb\lb} \Rt_{i\jb k\lb}$ we use the Gauss equation (Lemma~\ref{lemma:Gauss-equation}). We obtain $L$ written purely in terms of geometric quantities,
  \begin{multline*}
    L = - \frac18 \Deltat f +\frac14 \nablat_H f + f \Big( - \frac{1}{8}R + \frac{3}{32}{\Rt} -\frac{1}{8}\bracket{H,H} + \frac{1}{24}\bracket{h,h}\Big) \\+ \frac14 \div(\nabla f - KN + fKH).
  \end{multline*}
  Integrating over $\Sigma$, the divergence gives us boundary terms and we obtain 
  \begin{multline*}
    I_1 = \int_\Sigma \Big[- \frac18 \Deltat f +\frac14 \nablat_H f + f \Big( - \frac{1}{8}R + \frac{3}{32}{\Rt} -\frac{1}{8}\bracket{H,H} + \frac{1}{24}\bracket{h,h}\Big)\Big]dm \\
    + \int_{\partial\Sigma}  \frac14 \bracket{\nabla f - KN + fKH, \nu}d\sigma.
  \end{multline*}
  Here $\nu$ is the outer normal and $\sigma$ is the volume form induced by $g$ on $\partial\Sigma$.

  \paragraph{Term $I_2$.}
  We have 
  \begin{equation*}
    \abs{I_2(\eps)} \le \int_Y \abs{\calR(\eps,y)}\,dy \le C\int_Y\norm{r(\cdot,y)}_{W^{4,\infty}(B(x(y), \delta))} + \norm{r(\cdot,y)}_{L^1(X)},
  \end{equation*}
  which we write as $\abs{I_2(\eps)} \le C (\norm{r}_{L^1_yW^{4,\infty}_x(\Sigma_\delta)} + \norm{r}_{L^1(X\times Y)})$.
\end{proof}

%% file: quantitative_remainder.tex
% !TEX root = main.tex

\section{Laplace method with quantitative remainder}
\label{SecQuantitativeLaplace}

\subsection{Statement of the result}

Recall that the classical Laplace method consists in studying the behavior as $\eps\to 0^+$ of the integral
\[
  \int_\Rd e^{-u(x)/\eps} \,r(x)dx\,.
\]
In this section we prove a Laplace formula with explicit zeroth- and first-order terms and a quantitative remainder.

First let us introduce some notations. When $f$ is a function defined over $\Rd$, the ``norm of its $k$th derivative'' is defined as
\begin{equation*}
  \abs{D^kf} \coloneqq \sum_\alpha \abs{\partial_{1}^{\alpha_1}\partial_{2}^{\alpha_2}\dots \partial_{d}^{\alpha_d}f},
\end{equation*}
where the sum runs over all multi-indices $\alpha$ such that $\alpha_1+\alpha_2+\dots+\alpha_d=k$. We also denote 
\begin{equation} \label{eq:def-norm-derivatives-upto}
  \abs{D^{\le k}f} \coloneqq \sum_{j=0}^k\abs{D^jf}.
\end{equation}
For a fixed $x\in \Rd$, the Taylor remainder of $f$ about $x$ is defined as the function
\begin{equation} \label{eq:def-taylor-remainder}
  R_nf(z)=f(x+z)-\sum_{k=0}^{n-1}\frac{1}{k!} \partial_{i_1\dots i_k}f(x)\,z^{i_1}\dots z^{i_k}\,,
\end{equation}
for $n\ge 1$. Here $z^i$ denotes the $i$th component of $z$, the indices $i_j$ run from $1$ to $d$ and following the Einstein summation convention the sum over $i_1,\dots,i_k$ is not explicitly written. 

We set $G(z)=(2\pi)^{-d/2} e^{-\frac{1}{2}\abs{z}^2}$ and define
\[
  G_\tau(z) = \tau^{-d/2}G(z/\sqrt{\tau}),
\]
whenever $\tau>0$.
We also introduce the convolution kernel
\[
  K(z) = \frac{e^{-\abs{z}^2/4}}{\abs{z}^{d-1}}\,,
\]
and set for any $\tau>0$,
\begin{equation} \label{eq:def-Keps:1}
  K_\tau(z)=\tau^{-d/2}K(z/\sqrt{\tau}).
\end{equation}
Note that $K\in L^1(\Rd)$ and that $\norm{K_\tau}_{L^1} = \norm{K}_{L^1}$.

We are now ready to state our quantitative version of the Laplace expansion at first order, although the strategy of the proof easily supports the extension to higher orders. 
We initially consider the following assumptions on $u$ (but see Assumption~\ref{ass:ux-localized} and Corollary~\ref{cor:quantitativelaplace} below):\\

\begin{assumption}\label{ass:ux} 
  \leavevmode
  \begin{enumerate}[(i)]
    \item \label{ass:ux:regularity}$u\in C^6(\Rd)$ and $\abs{D^ku}\in L^\infty(\Rd)$ for $3\le k\le 6$. 
    \item \label{ass:ux:bound-below} There exist $\lambda>0$ and a point $\xb\in\Rd$ such that $u(\xb)=0$ and 
    \[
      u(x)\ge \frac\lambda 2\abs{x-\xb}^2\,.
    \]
  \end{enumerate}
  
\end{assumption}

Note that~\ref{ass:ux:bound-below} implies that $u(x)\ge 0$ and that $u$ attains its minimum value $0$ at a unique point, $\xb$.

\begin{theorem} \label{thm:quantitativelaplace} 
  Let $u$ be a function satisfying Assumption~\ref{ass:ux} and let $r\in C^4(\Rd)$. Then there exists $C>0$ such that for all $\eps>0$,   
  \begin{multline*}
    \int_{\Rd} \frac{e^{-u(x)/\eps}}{(2\pi\eps)^{d/2}} \,r(x)dx = \\\frac{1}{\sqrt{\det[u_{ij}]}}\bigg[r + \eps\Big(\frac 12 u^{ij}\partial_{ij}r-\frac 12 u_{jk\ell}u^{ij}u^{k\ell}\partial_ir +
      \frac 18 r\,u_{ijk}u_{\ell mn}u^{ij}u^{k\ell}u^{mn}\\+\frac{1}{12} r\, u_{ijk}u_{\ell mn} u^{i\ell} u^{jm} u^{kn} - \frac 18 r\, u_{ijk\ell}u^{ij}u^{k\ell}\Big)\bigg]+ \eps^2\,\mathcal{R}(\eps)\,,
  \end{multline*}
  where the right-hand side is evaluated at $x=\xb$, and where
  \[
    \abs{\mathcal{R}(\eps)}\le  C\,\Big[\abs{D^{\le 2}r}(\xb) + (K_{\eps/\lambda} \!*\! \abs{D^{\le 4}r})(\xb)\Big]\,.
  \]
  The constant $C$ only depends on $d$, $\lambda$ and $\norm{D^ku}_{L^{\infty}(\Rd)}$ for $3\le k\le 6$. The quantity $\abs{D^{\le 4}r}$ is defined by~\eqref{eq:def-norm-derivatives-upto} and $K_{\eps/\lambda}$ is defined by~\eqref{eq:def-Keps:1}.

\end{theorem}

\begin{remark}
  In particular if $r\in W^{4,\infty}(\Rd)$, where $W^{4,\infty}(\Rd)$ stands for the usual Sobolev space of functions with four derivatives in $L^\infty$, then the remainder can be taken independent of $\eps$,
  \begin{equation} \label{eq:remainer-infty}
    \abs{\calR(\eps)} \le C\norm{r}_{W^{4,\infty}(\Rd)}\,.
  \end{equation}
\end{remark}

We now proceed to localize Theorem~\ref{thm:quantitativelaplace}. Indeed when doing a Laplace expansion, we expect that the regularity and the behavior of the various quantities at play should mostly matter on a neighborhood of the minimizer $\xb$. We therefore fix an open subset $X\subset\Rd$ and consider functions $u$ and $r$ defined over $X$. In the following assumptions, $B(\xb, \delta)$ denotes the ball of radius $\delta$ and center $\xb$ and $\abs{\cdot}$ denotes the Euclidean norm of $\Rd$.

\begin{assumption}[Assumptions on $X$ and $u$]\label{ass:ux-localized}
  \leavevmode
  \begin{enumerate}[(i)]
    \item\label{ass:ux-localized:gen} $u$ is a measurable function over $X$, $u\ge 0$ and there exists $\xb\in X$ such that $u(\xb)=0$. 
    \item\label{ass:ux-localized:ball} There exists $\delta>0$ such that $B \coloneqq B(\xb, \delta) \subset X$.
    \item\label{ass:ux-localized:regularity} $u\in C^6(B)$.
    \item\label{ass:ux-localized:bound-below}

    There exists $\lambda>0$ such that
    \begin{align*}
      u(x) &\ge \frac\lambda 2\abs{x-\xb}^2 \quad\text{for all $x\in B$}, \\    
      u(x) &\ge \frac\lambda 2\delta^2 \quad\text{for all $x\in X \setminus B.$}
    \end{align*}
  \end{enumerate}
\end{assumption}

Assumption~\ref{ass:ux-localized}\ref{ass:ux-localized:bound-below} can be replaced by demanding that $D^2u(\xb)\ge \lambda I_d$ and that $u$ be bounded below by a strictly positive constant outside of $B$ (this constant is taken here as $\frac\lambda 2\delta^2$ but its precise value matters little).

We now state a local version of Theorem~\ref{thm:quantitativelaplace} taking $r\in W^{4,\infty}(X)\cap L^1(X)$ for simplicity.

\begin{corollary} \label{cor:quantitativelaplace}
  Let $u$ be a function satisfying Assumption~\ref{ass:ux-localized} and let $r\in W^{4,\infty}(B)\cap L^1(X)$. Then there exists $C>0$ such that for all $\eps>0$,
  \begin{multline*}
    \int_X \frac{e^{-u(x)/\eps}}{(2\pi\eps)^{d/2}} \,r(x)dx = \\\frac{1}{\sqrt{\det[u_{ij}]}}\bigg[r + \eps\Big(\frac 12 u^{ij}\partial_{ij}r-\frac 12 u_{jk\ell}u^{ij}u^{k\ell}\partial_ir +
      \frac 18 r\,u_{ijk}u_{\ell mn}u^{ij}u^{k\ell}u^{mn}\\+\frac{1}{12} r\, u_{ijk}u_{\ell mn} u^{i\ell} u^{jm} u^{kn} - \frac 18 r\, u_{ijk\ell}u^{ij}u^{k\ell}\Big)\bigg]+ \eps^2\,\mathcal{R}(\eps)\,,
  \end{multline*}
  where the right-hand side is evaluated at $x=\xb$, and where
  \[
    \abs{\mathcal{R}(\eps)}\le  C\,\Big[\norm{r}_{W^{4,\infty}(B)} + \norm{r}_{L^1(X)}]\,.
  \]
  The constant $C$ depends on $d$, $\lambda$, $\delta$ and $\norm{D^ku}_{L^{\infty}(B)}$ for $3\le k\le 6$.
\end{corollary}

\subsection{Proof of Theorem~\ref{thm:quantitativelaplace}}

We start with a lemma providing upper bounds on certain types of Gaussian integrals of $R_nf$ in terms of the norm of the $n$th derivative of $f$. 

\begin{lemma}\label{lemma:taylor-remainder-convolution}
  There exists a constant $C=C(d,n,k)$ such that     
    \[
        \int_{\Rd}\abs{z}^k \abs{R_nf(z)} \,G_\tau(z)dz \le C \,\tau^{\frac{k+n}{2}} \, (K_\tau * \abs{D^nf})(x)\,,
    \]
    for any $\tau>0$ and any integers $k\ge 0$ and $n\ge 1$. 
\end{lemma}

\begin{proof}
  By the Taylor remainder theorem applied to $s\mapsto f(x+sz)$ we have 
  \[
      R_nf(z)=\int_0^1\frac{(1-s)^{n-1}}{(n-1)!}\partial_{i_1\dots i_n}f(x+sz)z^{i_1}\dots z^{i_n}\,ds\,.
  \]
  By Jensen's inequality, 
  \[
      \abs{R_nf}(z) \le \int_0^1\frac{1}{(n-1)!}(1-s)^{n-1}\abs{D^nf}(x+sz)\abs{z}^{n}\,ds\,,
  \]
  and we further bound $(1-s)^{n-1}$ by $1$. 
  Therefore 
  \[
      \int_{\Rd}\abs{z}^k\abs{R_nf(z)} \, G_\tau(z)dz \le \int_{\Rd}\int_{0}^1 \frac{1}{(n-1)!}\abs{D^nf}(x+sz)\abs{z}^{k+n} G_\tau(z)\,dsdz\,.
  \]
  Doing sequentially the following changes of variables $z\leftarrow  sz$ and $s\leftarrow \frac{\abs{z}}{s\sqrt{\tau}} $, the right-hand side becomes
  \[
    \int_\Rd\frac{1}{(n-1)!}\abs{D^nf}(x+z) \frac{\tau^{\frac{k+n-1}{2}} }{\abs{z}^{d-1}} \int_{\abs{z}/\sqrt{\tau}}^\infty s^{k+n+d-2} \frac{e^{-s^2/2}}{(2\pi)^{d/2}}\,ds\,dz\,.
  \]
  Write now $e^{-s^2/2} = e^{-s^2/4}e^{-s^2/4} \le e^{-\abs{z}^2/4\tau} e^{-s^2/4}$ when $s\ge \abs{z}/\sqrt{\tau}$, and bound the $s$ integral by the integral over $(0, \infty)$. This results in 
  \begin{multline*}
    \int_{\Rd}\abs{z}^k\abs{R_nf(z)} \, G_\tau(z)dz \le \frac{1}{(n-1)!}\int_0^\infty s^{k+n+d-2} \frac{e^{-s^2/4}}{(2\pi)^{d/2}}\,ds \\ \int_\Rd \abs{D^nf}(x+z) \tau^\frac{k+n}{2} K_\tau(z)\,dz.
  \end{multline*}
 We obtain the desired result with the constant $C(d,n,k) = \int_0^\infty s^{k+n+d-2} \frac{e^{-s^2/4}}{(2\pi)^{d/2}}\,ds$, finite since $k+n+d-2\ge 0$. 
\end{proof}

We will also compute explicitly certain Gaussian moments. We therefore recall Isserlis' formula (see~\cite{GVK025155202} which also contains an application to a formal Laplace expansion).

\begin{lemma}[Isserlis' formula] \label{lemma:isserlis}
  Let 
  \[    
    p_\eps(z) = \sqrt{\smash[b]{\det[u_{ij}(\xb)]}} \frac{e^{-\frac{1}{2\eps} u_{ij}(\xb)z^iz^j}}{(2\pi\eps)^{d/2}}
  \]
  be the Gaussian density with zero mean and covariance matrix $[\eps u^{ij}(\xb)]$, and fix indices $1\le i_k\le d$ for $k=1,\dots,2n$.
Then
    \[
        \int_{\Rd}z^{i_1}\dots z^{i_{2n}} \,p_\eps(z)dz=\eps^{n}\sum_P \prod_{\{a,b\}\in P}u^{i_ai_b}(\xb)\,,
    \]
    where the sum runs over all the partitions $P$ of $\{1,\dots,2n\}$ into pairs $\{a,b\}$.
\end{lemma} 

We are now ready to prove Theorem~\ref{thm:quantitativelaplace}.

\begin{proof}[Proof of Theorem~\ref{thm:quantitativelaplace}]
Let $\calU$ be the class of functions $u$ satisfying Assumption~\ref{ass:ux}. It is easy to check that $\calU$ is convex (in fact it is a convex cone). For the entirety of the proof we fix a function $u\in\calU$, a function $r\in C^4(\Rd)$ as well as $\eps>0$. In view of Assumption~\ref{ass:ux} the unique minimizer of $u$ is denoted by $\xb$. The proof revolves around the functional
\[
    F_\eps(w)=V \int_{\Rd} \frac{e^{-w(x)/\eps}}{(2\pi\eps)^{d/2}}\,r(x) dx\,,
\]
defined for $w\in\calU$ and where we write 
\[
  V = \sqrt{\smash[b]{\det[u_{ij}(\xb)]}}\,.
\]
Since $V$ depends on $u$ and not $w$, it is therefore constant throughout the proof. We choose to include it in $F_\eps$ to avoid writing many $\det[u_{ij}(\xb)]$ terms later on.
We also set 
\[
  u_0(x)=\frac 1 2 u_{ij}(\xb)(x^i-\xb^i)(x^j-\xb^j).
\]
Note that $u_0\in\calU$ and that the zeroth, first and second-order derivatives of $u_0$ and $u$ coincide at $\xb$. The main idea of the proof is then to estimate the desired quantity $F_\eps(u)$ by its Taylor expansion at $u_0$, in the form
\begin{multline}\label{eq:taylor-expansion-Feps}
    \Big\lvert F_\eps(u)-F_\eps(u_0)-\delta F_\eps(u_0)(u-u_0)-\frac 12 \delta^2\!F_\eps(u_0)(u-u_0)^{\otimes 2} - \frac 16 \delta^3\!F_\eps(u_0)(u-u_0)^{\otimes 3}\Big\rvert \le \\\frac{1}{24}\sup_{w\in\calU}\abs{\delta^4\!F_\eps(w)(u-u_0)^{\otimes 4}}\,.
\end{multline}
Here it is important that $\calU$ is convex. In the preceding formula, we denoted the $n$-th directional derivative of $F_\eps$ in direction $h$ by 
\[
  \delta^n\!F_\eps(w)(h)^{\otimes n} = \left.\frac{d^n}{dt^n}\right\lvert_{t=0} F_\eps(w+th)\,.
\]
The notation $(h)^{\otimes n}$ stands for $(h, \dots, h)$ ($n$ times).

Let us take a look at the directional derivatives of $F_\eps$. They can be easily computed: for $n\ge 1$,
\begin{equation*}
    \delta^n\!F_\eps(w)(h)^{\otimes n} = (-\eps)^{-n} V \int_{\Rd}  h(x)^n \frac{e^{-w(x)/\eps}}{(2\pi\eps)^{d/2}}\,r(x) dx\,.
\end{equation*} 

Since this proof makes repeated use of \emph{Taylor remainders}, we recall that for a function $f$ defined over $\Rd$ we denote $R_nf$ the Taylor remainder of order $n$ at the point $\xb$, see~\eqref{eq:def-taylor-remainder}. Thus $R_0f(z)=f(\xb+z)$, $R_1f(z)=f(\xb+z)-f(\xb)$, 
\[
  R_2f(z)=f(\xb+z)-(f(\xb)+\partial_if(\xb)z^i)
\]
(with the implied sum of repeated indices), and
\[
  R_3f(z)=f(\xb+z)-(f(\xb)+\partial_if(\xb)z^i-\frac 12 \partial_{ij}f(\xb)z^iz^j)\,.
\]
When $f$ is  controlled in an $L^\infty$-type Sobolev space, we  use the standard Taylor bound 
\begin{equation*}
    \abs{R_nf(z)}\le \frac{1}{n!}\norm{D^nf}_{L^{\infty}(\Rd)}\abs{z}^n\,.
\end{equation*}
We also need the identity
\begin{equation}  \label{eq:identity-Rnf}
  R_nf(z)=\frac{1}{n!}\partial_{i_1\dots i_n}f(\xb)z^{i_1}\dots z^{i_n} + R_{n+1}f(z)\,,
\end{equation}
which is immediate to derive, as well as higher-order generalizations of it. Finally to alleviate notation we write 
\[
    p_\eps(z) = V \frac{e^{-\frac{1}{2\eps} u_{ij}(\xb)z^iz^j}}{(2\pi\eps)^{d/2}}\,,
\]
which is the Gaussian density with zero mean and covariance matrix $[\eps u^{ij}(\xb)]$. The moments of $p_\eps$ can be computed by Isserli's formula recalled in Lemma~\ref{lemma:isserlis}.

We now proceed to evaluate the Taylor terms in~\eqref{eq:taylor-expansion-Feps}. For each term we compute exactly the terms of order $0$ and $1$ in $\eps$ and derive an explicit $O(\eps^2)$ bound for the remainder.

% % % % % % % % % % % % % % % % % % % 
%                                   %
%        F I R S T    T E R M       %
%                                   %
% % % % % % % % % % % % % % % % % % % 

\paragraph{First term $F_\eps(u_0)$.}
The first term is
\[
    F_\eps(u_0)=\int_{\Rd} r(\xb + z)  \, p_\eps(z)dz \,.
\]
We start by expanding $r(x)$ as a Taylor sum up to $O(\eps^2)$ terms. We thus need a third-order Taylor approximation on $r$,
\[
  r(\xb\!+z) = r(\xb)+\partial_ir(\xb)z^i+\frac 12 \partial_{ij}r(\xb)z^iz^j + \frac 16 \partial_{ijk}r(\xb) z^iz^jz^k + R_4r(z)\,.\]
By symmetry in the Gaussian integral all the odd-order terms cancel and we are left with
\[
    F_\eps(u_0)= \int_{\Rd} \Big[r(\xb) + \frac 12 \partial_{ij}r(\xb)z^iz^j + R_4r(z) \Big]  \, p_\eps(z)dz\,.
\]
By expanding the brackets we obtain three integrals. The first one sums up to $r(\xb)$. The second and third integrals are denoted $I_1$ and $I_2$ respectively. For $I_1$ we use the Isserlis formula given by Lemma~\ref{lemma:isserlis} to compute the Gaussian moment of order $2$ and find that 
\[
    I_1 = \frac \eps 2 u^{ij}\partial_{ij}r\Big|_{x=\xb}\,.
\]
The other integral $I_2=\int_{\Rd}R_4r(z) p_\eps(z)dz$ can be bounded by $O(\eps^2)$ terms as follows.
First Assumption~\ref{ass:ux} gives us a control on the Hessian of $u$ at $\xb$,
\[
  u_{ij}(\xb)z^iz^j\ge \lambda\abs{z}^2.
\] 
This implies an inequality we will reuse throughout the proof,
\begin{equation}\label{eq:upper-bound-p_eps}
  p_\eps(z) \le V \lambda^{-d/2} G_{\eps/\lambda}(z)\,,
\end{equation}
with $G_{\eps/\lambda}(z)=(2\pi\eps/\lambda)^{-d/2}e^{-\lambda/(2\eps)\abs{z}^2}$. Therefore
\[
  \abs{I_2} \le V\lambda^{-d/2}\int_{\Rd} \abs{R_4r(z)} G_{\eps/\lambda}(z)\,dz\,.
\]
Bounds on this type of integrals are provided by Lemma~\ref{lemma:taylor-remainder-convolution}. Here it gives us
\[
    \int_{\Rd} \abs{R_4r(z)} G_{\eps/\lambda}(z)\,dz \le \eps^2\lambda^{-2}\, c(d) (K_{\eps/\lambda}*\abs{D^4r})(\xb)\,,
\] 
where $c(d)$ is a constant depending only on the dimension and $K_{\eps/\lambda}$ is the kernel defined by~\eqref{eq:def-Keps:1}.
In conclusion we showed
\begin{equation} \label{laplace-prop-proof:term1}
  F_\eps(u_0) = r +\frac\eps 2u^{ij}\partial_{ij}r + \eps^2 \calR_\eps\,,
\end{equation}
where the right-hand side is evaluated at $x=\xb$ and with 
\[
    \abs{\calR_\eps(\xb)} \le c(d)  V\lambda^{-d/2}\lambda^{-2} (K_{\eps/\lambda}*\abs{D^4r})(\xb)\,.
\]

% % % % % % % % % % % % % % % % % % % 
%                                   %
%      S E C O N D    T E R M       %
%                                   %
% % % % % % % % % % % % % % % % % % % 

\paragraph{Second term $\delta F_\eps(u_0)(u-u_0)$.}
The next term in the expansion~\eqref{eq:taylor-expansion-Feps} is 
\[
    \delta F_\eps(u_0)(u-u_0) = -\eps^{-1}\int_{\Rd} r(\xb+z) \big(u(\xb+z)-u_0(\xb+z)\big) \, p_\eps(z)dz.
\]
We write $r(\xb+z)=r + \partial_irz^i + \frac 12 \partial_{ij}rz^{i}z^{j} + R_3r(z)|_{x=\xb}$ and expand the sum against $u(\xb+z)-u_0(\xb+z)$. Since $u(\xb)=0$ and the first derivative also vanishes, $ u_i(\xb)=0$, the expression $u(\xb+z)-u_0(\xb+z)$ turns out to be exactly the Taylor remainder $R_3u(z)=u(\xb+z)-(u+u_iz^i+\frac 12 u_{ij}z^iz^j)(\xb)$. Then $R_3u(z)$ is successively expressed with more or fewer terms, depending on which $r$ term is in front of it, in order to obtain the desired order in $\varepsilon$,
\begin{align*}
  r(\xb+z) \, R_3u(z) = &r  \Big(\frac{1}{3!}u_{ijk}z^iz^jz^k + \frac{1}{4!}u_{ijk\ell}z^iz^jz^kz^{\ell} + \frac{1}{5!}u_{ijk\ell m}z^iz^jz^kz^\ell z^m \\+ R_6u(z)\Big)
   &+ \partial_irz^i \Big(\frac{1}{3!}u_{jk\ell}z^jz^kz^{\ell} + \frac{1}{4!}u_{jk\ell m}z^jz^kz^\ell z^m + R_5u(z)\Big) \\
   &+ \frac 12\partial_{ij}rz^iz^j \Big(\frac{1}{3!} u_{k\ell m}z^kz^\ell z^m +R_4u(z)\Big)
   + R_3r(z) \,R_3u(z)\,,
\end{align*}
where all the functions on the right-hand side are evaluated at $x=\xb$. Each expression enclosed in brackets is a different way to write $R_3u(z)$, as in~\eqref{eq:identity-Rnf}. After taking advantage of cancellations in the Gaussian integral due to symmetry we are left with 
\begin{multline*}
    \delta F_\eps(u_0)(u-u_0)= -\eps^{-1}\int_\Rd \Big(\frac{1}{4!} r u_{ijk\ell} + \frac{1}{3!} \partial_i ru_{jk\ell}\Big)(\xb) z^iz^jz^kz^\ell \, p_\eps(z)dz -\\
    \eps^{-1}\int_\Rd \Big(r R_6u(z)+\partial_i rz^i R_5u(z) + \frac 12\partial_{ij}rz^iz^jR_4u(z) + R_3r(z) \,R_3u(z)\Big)_{x=\xb} \, p_\eps(z)dz\,.
\end{multline*}
The first integral is denoted by $I_1$, it is a first-order term in $\eps$. The second integral is denoted $I_2$ and contains higher-order terms. Let us focus first on $I_1$. Using the Isserlis formula in Lemma~\ref{lemma:isserlis} we calculate (dropping $\xb$ for convenience)
\begin{align*}
    I_1 = -\eps\Big(\frac{1}{4!} r \,u_{ijk\ell} + \frac{1}{3!} \partial_ir\,u_{jk\ell}\Big)\big(u^{ij}u^{k\ell}+u^{i\ell}u^{jk}+u^{ik}u^{j\ell}\big)\,.
\end{align*}
The object $u_{ijk\ell}$ is invariant by a permutation of the indices $i,j,k,\ell$. As for $\partial_ir\partial_{jk\ell}u$ it is not totally symmetric in the indices but its invariance to permutation of the indices $j,k,\ell$ implies that the expression of the rightmost bracket can be simplified to $3u^{ij}u^{k\ell}$. Therefore 
\[
    I_1 = -\frac\eps 8 r\, u_{ijk\ell}\,u^{ij}u^{k\ell} - \frac\eps 2 \partial_ir \,u_{jk\ell}\,u^{ij}u^{k\ell}\,.
\]
We now turn our attention to 
\begin{multline*}
I_2 = -\eps^{-1}\int_{\Rd} \Big(r R_6u(z)+\partial_i rz^i R_5u(z) + \frac 12\partial_{ij}rz^iz^jR_4u(z) + \\ R_3r(z) \,R_3u(z)\Big)_{x=\xb} \, p_\eps(z)dz\,.
\end{multline*}
We proceed similarly as for $F_\eps(u_0)$. Expressions of $u$ are bounded in $L^\infty$ by $\abs{R_nu(z)}\le\frac{1}{n!}\normlinf{D^nu}\abs{z}^n$ and $p_\eps$ is bounded using~\eqref{eq:upper-bound-p_eps}. We obtain 
\begin{equation*}
  \abs{I_2}\le c\,\eps^{-1}V\lambda^{-d/2} \norm{D^3u}_{W^{3,\infty}} \int_{\Rd} \Big( \abs{z}^6\abs{D^{\le 2}r (\xb)} + \abs{z}^3\abs{R_3r(z)} \Big)  G_{\eps/\lambda}(z)\,dz\,,
\end{equation*}
for a numeric constant $c>0$.
The first term $\abs{z}^6\abs{D^{\le 2}r (\xb)}$ yields an elementary Gaussian moment and we use Lemma~\ref{lemma:taylor-remainder-convolution} for  $\abs{R_3r(z)}$. We obtain
\[
 \abs{I_2} \le \eps^2\,c(d)V\lambda^{-d/2}\lambda^{-3} \norm{D^3u}_{W^{3,\infty}(\Rd)} (\abs{D^{\le 2}r} + (K_{\eps/\lambda}*\abs{D^3r}))(\xb)\,.
\]
In conclusion we have proven 
\begin{equation} \label{laplace-prop-proof:term2}
  \delta F_\eps(u_0)(u-u_0) = -\frac\eps 8 r\, u_{ijk\ell}\,u^{ij}u^{k\ell} -\frac\eps 2 \partial_ir\, u_{jk\ell}\,u^{ij}u^{k\ell} + \eps^2 \calR_\eps\,,
\end{equation}
where the right-hand side is evaluated at $x=\xb$, and with 
\begin{equation*}
  \abs{\calR_\eps(\xb)} \le c(d)V\lambda^{-d/2}\lambda^{-3} \norm{D^3u}_{W^{3,\infty}(\Rd)} \Big[\abs{D^{\le 2}r}(\xb) + (K_{\eps/\lambda}*\abs{D^3r})(\xb)\Big]\,.
\end{equation*}

% % % % % % % % % % % % % % % % % % % 
%                                   %
%        T H I R D    T E R M       %
%                                   %
% % % % % % % % % % % % % % % % % % % 

\paragraph{Third term $\frac 12\delta^2\!F_\eps(u_0)(u-u_0)^{\otimes 2}$.}
The next term in~\eqref{eq:taylor-expansion-Feps} is
\[
    \frac 12\delta^2\!F_\eps(u_0)(u-u_0)^{\otimes 2} = \frac 12\eps^{-2}\int_{\Rd} r(\xb+z)\big(u(\xb+z)-u_0(\xb+z)\big)^2 \, p_\eps(z)dz.
\]
Following the line of proof for the previous $F_\eps$ terms, we write $r(\xb+z)= r + \partial_irz^i +  R_2r(z)$ and $u(\xb+z)-u_0(\xb+z)=R_3u(z)$. Then the term $r(\xb+z) (R_3u(z))^2$ is broken up as follows,
\begin{align*}
  &  r(\xb+z) \, R_3u(z) R_3u(z) = \\
  & \quad  r\, \frac{1}{3!} u_{ijk} z^iz^jz^k  \Big(\frac{1}{3!}u_{\ell mn}z^\ell z^mz^n + \frac{1}{4!}u_{\ell mns}z^\ell z^mz^nz^s + R_5u(z)\Big)\\
  & +r\, \frac{1}{4!}u_{ijk\ell} z^iz^jz^kz^\ell \Big(\frac{1}{3!}u_{mns}z^mz^nz^s + R_4u(z)\Big)\\
  & +r\, R_5u(z) R_3u(z) \\
  & +\partial_ir z^i \Big(\frac{1}{3!} u_{jk\ell}z^j z^kz^\ell \Big) \Big(\frac{1}{3!} u_{mns}z^mz^nz^s  + R_4u(z)\Big) \\
    &  + \partial_ir z^i \, R_4u(z) \, R_3u(z) \\
    &  + R_2r(z)\,  R_3u(z) \, R_3u(z)\,,
\end{align*}
where the right-hand side is evaluated at $x=\xb$ where needed. Under the Gaussian integral, we obtain first a moment of order $6$,
\[
  I_1 := \eps^{-2}\frac{1}{2\cdot (3!)^2} r\,u_{ijk}u_{\ell mn}\Big\rvert_{x=\xb}\int_{\Rd} z^iz^jz^kz^\ell z^mz^n \,p_\eps(z) dz\,.
\]
The moments of order $7$ vanish and all the other terms form various moments of order $8$ that we denote $I_2$. 
Let us first compute $I_1$. Using Isserlis' formula and symmetries of the partial derivatives we have 
{\begin{align*}
  I_1 &=\frac{\eps}{72}r\,u_{ijk}\,u_{\ell mn}\big(9u^{ij}u^{k\ell}u^{mn} +6 u^{i\ell}u^{jm}u^{kn} \big)\\
  &=\frac{\eps}{8}r\,u_{ijk}u_{\ell mn} \,u^{ij}u^{k\ell}u^{mn} +\frac{\eps}{12}r\,u_{ijk}\,u_{\ell mn}u^{i\ell}u^{jm}u^{kn}\,.
\end{align*}}

It remains to bound $I_2$. By following a similar line of proof as was done for the $\delta F_\eps(u_0)(u-u_0)$ term, we obtain
\[
  \abs{I_2} \le c\,\eps^{-2}V\lambda^{-d/2} \norm{D^3u}^2_{W^{2,\infty}} \int_{\Rd}\Big(\abs{z}^8\abs{D^{\le 1} r(\xb)} + \abs{z}^6\abs{R_2r(z)}\Big)\,G_{\eps/\lambda}(z)dz\,.
\]
By Lemma~\ref{lemma:taylor-remainder-convolution} we deduce
\[
  \abs{I_2} \le \eps^2\,c(d)V\lambda^{-d/2}\lambda^{-4} \norm{D^3u}^2_{W^{2,\infty}} (\abs{D^{\le 1}r} + (K_{\eps/\lambda}*\abs{D^2r}))(\xb)\,.
\]
In conclusion we have proven 
\begin{equation} \label{laplace-prop-proof:term3}
  \frac 12\delta^2 F_\eps(u_0)(u\!-\!u_0)^{\otimes 2}  = \frac{\eps}{8}r\,u_{ijk}u_{\ell mn} \,u^{ij}u^{k\ell}u^{mn} +\frac{\eps}{12}r\,u_{ijk}\,u_{\ell mn}u^{i\ell}u^{jm}u^{kn} + \eps^2 \calR_\eps \,,
\end{equation}
evaluated at $x=\xb$, and with
\begin{equation*}
  \abs{\calR_\eps(\xb)}
  \le c(d)V\lambda^{-d/2}\lambda^{-4} \norm{D^3u}^2_{W^{2,\infty}} \Big[\abs{D^{\le 1}r}(\xb) + (K_{\eps/\lambda}*\abs{D^2r})(\xb)\Big]\,.
\end{equation*}

% % % % % % % % % % % % % % % % % % % 
%                                   %
%      F O U R T H    T E R M       %
%                                   %
% % % % % % % % % % % % % % % % % % % 

\paragraph{Fourth term $\frac 16\delta^3\!F_\eps(u_0)(u-u_0)^{\otimes 3}$.}

We use similar arguments as for the previous terms to deal with
\[
    \delta^3\!F_\eps(u_0)(u-u_0)^{\otimes 3} = -\eps^{-3}\int_{\Rd} r(\xb+z)\big(u(\xb+z)-u_0(\xb+z)\big)^3 \, p_\eps(z)dz\,.
\]
This term is $O(\eps^2)$ so we simply bound it: we split
\begin{align*}
  & \quad r(\xb+z) (R_3u(z))^3 = \\
  & +r\,\frac{1}{3!}u_{ijk}z^iz^jz^k\frac{1}{3!}u_{\ell mn}z^\ell z^mz^n \Big(\frac{1}{3!}u_{abc}z^az^bz^c + R_4u(z)\Big)\\
  & +r\,\frac{1}{3!}u_{ijk}z^iz^jz^k R_4u(z)R_3u(z) \\
  & +r\,R_4u(z)R_3u(z) R_3u(z)\\
  & +R_1r(z) R_3u(z)R_3u(z)R_3u(z)\,.
\end{align*}
We end up with
\begin{equation} \label{laplace-prop-proof:term4}
  \delta^3\!F_\eps(u_0)(u-u_0)^{\otimes 3} = \eps^2 \calR_\eps(\xb)\,,
\end{equation}
with 
\[
  \abs{\calR_\eps(\xb)} \le \, c(d) V\lambda^{-d/2}\lambda^{-5}\norm{D^3u}^3_{W^{1,\infty}} \Big[\abs{r(\xb)} + (K_{\eps/\lambda} * \abs{Dr})(\xb)\Big]\,.
\]

% % % % % % % % % % % % % % % % % % % 
%                                   %
%    R E M A I N D E R   T E R M    %
%                                   %
% % % % % % % % % % % % % % % % % % % 

\paragraph{Remainder term $\delta^4\!F(w)(u-u_0)^{\otimes 4}$.}
The final step is to bound 
\begin{multline*}
  \delta^4\!F(w)(u-u_0)^{\otimes 4}=\\ \eps^{-4}V\int_\Rd r(\xb+z) \big(u(\xb+z)-u_0(\xb+z)\big)^4 \frac{e^{-w(\xb+z)/\eps}}{(2\pi\eps)^{d/2}}\,dz
\end{multline*}
uniformly over $w\in\calU$. Contrary to the previous terms we cannot use~\eqref{eq:upper-bound-p_eps} but Assumption~\ref{ass:ux}\ref{ass:ux:bound-below} gives us precisely what we need. Indeed for any $w$ satisfying Assumption~\ref{ass:ux}\ref{ass:ux:bound-below} we have $e^{-w(\xb+z)/\eps}\le e^{-\lambda\abs{z}^2/(2\eps)}$, and therefore
\[
  \frac{e^{-w(\xb+z)/\eps}}{(2\pi\eps)^{d/2}} \le \lambda^{-d/2} G_{\eps/\lambda}(z).
\]
From then we can finish up as for the previous $F_\eps$ terms. We bound 
\begin{align*}
    \abs{r(\xb+z) \, (R_3u(z))^4} \le c\, \normlinf{D^3u}^4 \abs{r(\xb+z)}\abs{z}^{12},
\end{align*}
which leads to
\begin{equation} \label{laplace-prop-proof:term5}
  \delta^4\!F_\eps(w)(u-u_0)^{\otimes 4} = \eps^{2} \calR_\eps(\xb)\,,
\end{equation}
with 
\[
    \abs{\calR_\eps(\xb)}\le  c(d)V\lambda^{-d/2}\lambda^{-6}\norm{D^3u}^4_{L^{\infty}} (K_{\eps/\lambda} * \abs{r})(\xb)\,.
\]

\paragraph{Putting everything together.}
We see that in each estimate \eqref{laplace-prop-proof:term1}, \eqref{laplace-prop-proof:term2}, \eqref{laplace-prop-proof:term3}, \eqref{laplace-prop-proof:term4} and \eqref{laplace-prop-proof:term5} the remainder term $\calR_\eps$ can be bounded above by 
\[
  c(d, \lambda, \norm{D^3u}_{W^{3,\infty}}) V \Big[\abs{D^{\le 2}r}(\xb) + (K_{\eps/\lambda} * \abs{D^{\le 4}r})(\xb)\Big]\,.
\]
Dividing by $V=\sqrt{\smash[b]{\det[u_{ij}(\xb)]}}$, the desired result follows.
\end{proof}

\begin{proof}[Proof of Corollary~\ref{cor:quantitativelaplace}]
  Let $B'$ denote the open ball $B(\xb,\delta/2)$. By standard arguments there exists a function $\chi\in C^\infty(\Rd)$ such that $\chi\ge 0$, $\chi=1$ on $B'$ and $\chi$ has support inside $B$. We then define 
  \[
    \begin{cases}
      \tilde r(x) = \chi(x) r(x) & \text{ for  }x\in B\\
      r(x)=0 & \text{ for }x\in \Rd\setminus B.
    \end{cases}
  \]
  Let $q(x)= \frac\lambda 2\abs{x-\xb}^2$. We similarly define 
  \[
    \begin{cases}
      \tilde u(x) = \chi(x)(u(x)-q(x)) + q(x) & \text{for } x\in B\\
      u(x)=q(x) &\text{ for }x\in \Rd\setminus B.
    \end{cases}
  \]

  We want to apply Theorem~\ref{thm:quantitativelaplace} with remainder in the form~\eqref{eq:remainer-infty}. Because of how $\tilde r$ was constructed it is easy to check that $\tilde r\in W^{4,\infty}(\Rd)$ with $\norm{\tilde r}_{W^{4,\infty}(\Rd)}\le C \norm{r}_{W^{4,\infty}(B)}$, for a constant $C$ that depends on $\delta$ and $d$. It is also straightforward to check that $\tilde u$ satisfies Assumption~\ref{ass:ux}. Also note that $\tilde r$ (resp. $\tilde u$) only depends on the values of $r$ (resp. $u$) inside $B$. 

  Our goal is then to replace the integral over $X$, namely $\int_X \frac{e^{-u(x)/\eps}}{(2\pi\eps)^{d/2}} \,r(x)dx$ by the integral over $\Rd$, namely $\int_\Rd \frac{e^{-\tilde u(x)/\eps}}{(2\pi\eps)^{d/2}} \,\tilde r(x)dx$. We bound 
  \begin{multline} \label{eq:proof-cor-quant-eq1}
    \Big\lvert\int_X \frac{e^{-u(x)/\eps}}{(2\pi\eps)^{d/2}} \,r(x)dx-\int_\Rd \frac{e^{-\tilde u(x)/\eps}}{(2\pi\eps)^{d/2}} \,\tilde r(x)dx\Big\rvert \le \\
    \int_X \Big\lvert\frac{e^{-u(x)/\eps}}{(2\pi\eps)^{d/2}} \,r(x)-\frac{e^{-\tilde u(x)/\eps}}{(2\pi\eps)^{d/2}} \,\tilde r(x)\Big\rvert dx + \int_{\Rd\setminus X} \frac{e^{-\tilde u(x)/\eps}}{(2\pi\eps)^{d/2}} \,\abs{\tilde r(x)}\,dx\,.
  \end{multline}
  In the first term on the right-hand side, the integrand vanishes on $B'$. Outside of $B'$, we simply bound the difference by the sum, thus
  \[
    \int_X \Big\lvert\frac{e^{-u(x)/\eps}}{(2\pi\eps)^{d/2}} \,r(x)-\frac{e^{-\tilde u(x)/\eps}}{(2\pi\eps)^{d/2}} \,\tilde r(x)\Big\rvert dx \le
     \int_{X\setminus B'} \frac{e^{-u(x)/\eps}}{(2\pi\eps)^{d/2}} \,\abs{r(x)}+\frac{e^{-\tilde u(x)/\eps}}{(2\pi\eps)^{d/2}} \,\abs{\tilde r(x)}dx\,.
  \]
  Combining with~\eqref{eq:proof-cor-quant-eq1} we obtain 
  \begin{multline*}
    \Big\lvert\int_X \frac{e^{-u(x)/\eps}}{(2\pi\eps)^{d/2}} \,r(x)dx-\int_\Rd \frac{e^{-\tilde u(x)/\eps}}{(2\pi\eps)^{d/2}} \,\tilde r(x)dx\Big\rvert \le \\
    \int_{X\setminus B'} \frac{e^{-u(x)/\eps}}{(2\pi\eps)^{d/2}} \,\abs{r(x)}
     + \int_{\Rd\setminus B'} \frac{e^{-\tilde u(x)/\eps}}{(2\pi\eps)^{d/2}} \,\abs{\tilde r(x)}\,dx\,.
  \end{multline*}

  On $X\setminus B'$ we have $u(x)\ge\frac\lambda 2\delta^2$ and on $\Rd\setminus B'$ we have that $\tilde u(x)\ge \frac\lambda 2\abs{x-\xb}^2\ge \frac\lambda 2\delta^2$ . Thus 
  \begin{multline} \label{eq:proof-cor-quant-eq2}
    \Big\lvert\int_X \frac{e^{-u(x)/\eps}}{(2\pi\eps)^{d/2}} \,r(x)dx-\int_\Rd \frac{e^{-\tilde u(x)/\eps}}{(2\pi\eps)^{d/2}} \,\tilde r(x)dx\Big\rvert \le \\
    \le (2\pi\eps)^{-d/2} e^{-\lambda \delta^2/(2\eps)} \big(\norm{r}_{L^1(X)} + \norm{\tilde r}_{L^1(\Rd)} \big)\,.
  \end{multline}
  We have  $\norm{\tilde r}_{L^1(\Rd)}\le C \norm{r}_{L^1(B)}$ and thus there exists a constant $C(\lambda, \delta,d)$ such that for all $\eps>0$,  the right-hand side of~\eqref{eq:proof-cor-quant-eq2} can be bounded by 
  \[
    C \eps^2\norm{r}_{L^1(X)}.
  \]
  On the other hand, we now perform the Laplace expansion of 
  \[
    \int_\Rd \frac{e^{-\tilde u(x)/\eps}}{(2\pi\eps)^{d/2}} \,\tilde r(x)dx\,.
  \]
  The zeroth and first-order terms consist of quantities depending on derivatives of $\tilde u$ and $\tilde r$ evaluated at $\xb$. Since $\tilde u,u$ and $\tilde r,r$ are equal on $B'$ we may remove the tilde. The remainder, in the form~\eqref{eq:remainer-infty} can be bounded by 
  \[
    \abs{\calR(\eps)} \le C\norm{\tilde r}_{W^{4,\infty}(\Rd)} \le C\norm{r}_{W^{4,\infty}(B)}\,.
  \]

\end{proof}

%% file: applications.tex
% !TEX root = main.tex
\section{Examples}
\label{sec:applications}
In this section, we discuss some examples that naturally arise without discussing the regularity assumptions imposed in Section \ref{sec:geometric_laplace} and only focusing on the geometric formula of the first-order term. We treat various cases in which simplifications of different kinds occur. The first case recovers the standard Laplace formula using a translation invariant cost. We treat the case of a parametrix for the heat kernel for which the second fundamental form vanishes. The likelihood in Bayesian modelling exhibits a flat geometry of both on $\Sigma$ and on $X \times Y$. Lastly, the Fenchel--Young gap on Euclidean spaces induces a Hessian geometry on the surface $\Sigma$ in a flat ambient space, and up to a change of variable, we obtain a remarkably simple formula.

\subsection{The translation invariant cost and the usual Laplace formula} \label{sec:translation-invariant-cost}

Let $U\colon\R^d \to \R$ be a strongly convex and nonnegative function such that $U(0) = 0$, the uniquely attained minimum of the function. 
We choose $X$ to be any bounded open subset of $\Rd$ (say, a ball), $Y=\Rd$ and we consider the extended function $u\colon X \times \R^d \to \R$, $u(x,y) = U(x - y)$. This corresponds to a translation cost often used in optimal transportation. 

The function $u$ satisfies the conditions of Section \ref{sec:geometric_laplace} and the vanishing set is the diagonal $\Sigma = \{ (x,x)  : x\in X\}$. We consider the density $r(dx,dy)= F(x - y)dxdy$, where $dx$ and $dy$ stand for the Lebesgue measures on $X$ and $Y$, respectively. 
Note that all the geometric quantities evaluated on $\Sigma$ are constant on $\Sigma$. Then the standard Laplace method is given a geometric formulation by our formula. 

Writing $U=U(z)$, we denote  
\[
    U_i=\frac{\partial U}{\partial z^i}\,, \quad U_{ij}=\frac{\partial^2 U}{\partial z^i\partial z^j}\,, 
\]
and so on. Then $u(x,y)=U(x-y)$ implies that $u_i=U_i$, $u_\ib=-U_i$, etc, and therefore  we can associate in this way a barred index on $u(x,y)$ with its non barred version on $U(z)$. The geometric quantities read as follows:
\begin{enumerate}[(a)]
\item The Kim--McCann metric on the product space is $\gt = \frac 12 \begin{pmatrix} 0 & D^2U \\ D^2U & 0\end{pmatrix}$. It is non degenerate since the Hessian matrix $D^2U=(U_{ij})$ is non singular. The volume form is $ \mt = \det(D^2U)$.
\item On $\Sigma$, $g_{ij} = U_{ij}(0)$ which does not depend on $x,y$ and $m =\sqrt{ \operatorname{det}(D^2U(0))} = \sqrt{ \mt}$. Since this metric is constant the Christoffel symbols and curvature vanish.
\item $f(x,y)\coloneqq \frac{dr(x,y)}{d\mt(x,y)}=\frac{F(x-y)}{\det D^2U(x-y)}$.
\item $ \Deltat f = -4U^{i j}\partial_{ij} \Big(F/\det D^2U\Big)$.
\item $\Gammat_{ij}^k = U^{k\ell} U_{\ell ij} \mbox{ and }   \Gammat_{\bari \jb}^\kbar = -U^{k\ell} U_{\ell ij}$. Since $\Gamma^k_{ij}=0$, \eqref{eq:Gamma-Gammat-h} gives $h^k_{ij} = U^{k\ell}(0) U_{\ell ij}(0) $.
\item $2 \Rt_{i\jb \kb \ell} = u_{i\jb \kb \ell} - u_{i \ell \sbar }u^{\sbar t}u_{t \jb \kb} = U_{ijk\ell} - U_{i \ell s}U^{s t}U_{tjk} \,.$
\item $\Rt  = 8 u^{i\kb} u^{\jb \ell} \Rt_{i\jb\kb \ell} = 4 U^{ik} U^{j\ell} (U_{ijk\ell} - U_{i \ell s}U^{s t}U_{tjk})$ and $R = 0$.
\item The norms of the second fundamental form and the mean curvature are given by 
\begin{align*}
 & \bracket{h,h} = - U_{ijk}U_{\ell mn}U^{i\ell}U^{jm}U^{kn}\,,
\\
   & \bracket{H,H} = - U_{ijk}U_{\ell mn}U^{ij}U^{k\ell}U^{mn}\,,
\end{align*}
where all the objects on the right-hand side are evaluated at $z=0$.
\end{enumerate}
Finally we see that since all the quantities depending on the geometry and $F$ are constant on $\Sigma$, they are constant on $\partial\Sigma$ and thus the boundary terms vanish. 

Let us instantiate this example in one dimension. 
We have $h = \frac {U^{(3)}}{U^{(2)}}$, $H =\frac {U^{(3)}}{(U^{(2)})^2} $, $\Rt = \frac{4}{(U^{(2)})^2}\left( U^{(4)} - \frac{(U^{(3)})^2}{U^{(2)}}\right)$, $\langle h,h\rangle = \langle H,H\rangle = -\frac{(U^{(3)})^2}{(U^{(2)})^3}$ and $-\frac 18 \Deltat (\frac {F}{U^{(2)}}) = \frac{F^{(2)}}{2(U^{(2)})^2} - \frac{F' U^{(3)}}{(U^{(2)})^3} + \frac{F (U^{(3)})^2}{(U^{(2)})^4} - \frac{Fu^{(4)}}{2(U^{(2)})^3}$. 
The Laplace double integral 
\[
    \int_X\int_{\R}\frac{e^{-U(x-y)/\eps}}{(2\pi\eps)^{1/2}}F(x-y)\,dxdy
\]
reduces to an integral over $z=x-y\in\R$ multiplied by the volume of $X$, and
we get
\begin{multline*}
    \int_{\R}\frac{e^{-U(z)/\eps}}{(2\pi\eps)^{1/2}}F(z)\,dz  =  \frac{F(0)}{\sqrt{U^{(2)}}(0)} \,+  \eps \left( \frac{1}{2\sqrt{U^{(2)}(0)}}\left(\frac{F}{U^{(2)}}\right)''(0)\right) \\+ \eps\frac {\sqrt{U^{(2)}}(0)H}2 \left( \frac {F}{U^{(2)}}\right)'+\eps\Big( \frac{3}{32}{\Rt} - \frac{1}{12}\bracket{h,h} \Big)\frac{F(0)}{\sqrt{U^{(2)}(0)}} + O(\eps^2)\,.    
\end{multline*}
We obtain, as in \cite[Chapter 6]{Bender1999},
\begin{multline*}
\int_\RR e^{-U(z)/\epsilon}F(z)\,dz = \sqrt{\frac{2\pi\epsilon}{U^{(2)}(0)}}\\ \Bigg[F(0)+ \epsilon \left( \frac{F^{(2)}}{2U^{(2)}} - \frac{F U^{(4)}}{8 (U^{(2)})^2} - \frac{F'
    U^{(3)}}{2 (U^{(2)})^2} + \frac{5 F (U^{(3)})^2}{24 (U^{(2)})^3}  \right) + O(\epsilon^2) \Bigg]\,.
\end{multline*}

The corresponding formula in higher dimension is more difficult to find in the literature\footnote{It can be found in \cite[Chapter 6, Lemma 6.5.3]{Kolassa1997}, but with errors in the renormalization factor and in some of the signs of the coefficients.}; it takes the form
\begin{multline*}
\int_{\RR^d} e^{-U(z)/\epsilon}F(z) \,dz = \frac{(2\pi\epsilon)^{d/2}}{\sqrt{\operatorname{det}(U_{ij})}}\Bigg[F(0)\\+ \epsilon \Big[\frac12 U^{ij}\partial_{ij}F- \frac12\partial_i F U_{jk\ell}U^{ij}U^{k\ell} +F\Big(-\frac18 U_{ijk\ell} U^{ij}U^{k\ell} \\ +\frac{1}{12}U_{ijk}U_{\ell mn}U^{i\ell}U^{jm}U^{kn} + \frac18U_{ijk}U_{\ell mn}U^{ij}U^{k\ell}U^{mn}\Big)\Big] \Bigg] + O(\epsilon^2)\,.
\end{multline*}
It is the expression given in Theorem~\ref{thm:quantitativelaplace}.
To derive it from the translation-invariant cost, we use the additional formulas
\begin{multline*}
- \frac18\Deltat (r/\mt) = \frac{1}{m}\Big[\frac12 U^{ij}\partial_{ij}F - \partial_i F U_{jk\ell}U^{ij}U^{k\ell} \\ + F\Big(-\frac12 U_{ijk\ell} U^{ij}U^{k\ell}+\frac12U_{ijk}U_{\ell mn}U^{i\ell}U^{jm}U^{kn}+\frac12U_{ijk}U_{\ell mn}U^{ij}U^{k\ell}U^{mn}\Big)\Big]
\end{multline*}
and
\begin{equation*}
\frac14\nablat_{\!H}(r/\mt)= \frac1{2m}U^{ij}U^{km}U_{ijm}\left(\partial_kF-FU_{kst}U^{st}\right)\,.
\end{equation*}

\subsection{Small-time limit of the heat kernel}\label{SecHeatKernel}

Let $(M,g)$ be a Riemannian manifold without boundary. Take $X=Y=M$ and consider the function $u(x,y)=\frac12 d^2(x,y)$, where $d$ is the Riemannian distance on $M$. We note that $u$
is symmetric,  $u(x,y)=u(y,x)$ and vanishes on the diagonal, i.e. $\Sigma=\{(x,x) : x\in M\}$. Then the restriction of the Kim--McCann metric to $\Sigma$ is precisely $g$ and therefore $\Sigma$ is an isometric copy of $M$~\cite[Example 3.6]{kim2007continuity}. 

Let us fix some notation. Since $y(x)=x$, the matrix $\partial_iy^\ib$ is the identity and thus provides a way to identify a barred index with an unbarred index. For functions of one variable (either $x$ or $y$) we will therefore be allowed to identify $i \leftrightarrow \ib$, $j \leftrightarrow \jb$, and so on.

In this setting several simplifications occur. On $\Sigma$ we have:

\begin{enumerate}[(a)]
    \item $g_{ij}(x)=u_{ij}(x,x)=-u_{i\jb}(x,x)$;
    \item $\Gammat^k_{ij}=\Gammat^\kb_{\ib\jb} = \Gamma^k_{ij}$;
    \item $h=0$, $H=0$;
    \item $\mt=m^2$, $\partial_i\log\mt=\partial_i\log m$ and $\partial_{i\jb}\log\mt=-\Rt_{i\jb}$.
\end{enumerate}

We also note that (b)--(d) hold in greater generality, whenever $u$ is a symmetric function that vanishes on the diagonal.

Consider now the heat equation on $M$,
\[
    \partial_tq = \Delta q,
\]
where $\Delta$ is the Laplacian on $(M,g)$ acting on scalar functions. The \emph{heat kernel} is the function $p_t(x,y)$ giving the solution at time $t>0$ from an initial condition, 
\[
    q_t(y) = \int_M p_t(x,y)q_0(x)\,dm(x)\,.
\]
Here we integrate against the Riemannian volume form $m$. The heat kernel has well-known small time asymptotics of the following form~\cite{minakshisundaram_pleijel_1949},
\begin{equation*}
p_t(x,y) = \frac{e^{-d(x,y)^2/4t}}{(4\pi t)^{d/2}} \sum_{k = 0}^\infty t^k \Phi_k(x,y)\,,
\end{equation*}
where the functions $\Phi_k$ (sometimes called ``Hadamard coefficients'') are solutions to some particular ``transport'' partial differential equations, see also \cite{rosenberg1997,Chavel1984EigenvaluesIR}.  
Let us also remark that in several applications, the reverse point of view is rather taken, \textit{i.e.} one solves the heat equation to obtain an estimation of the distance squared \cite{Crane}.

Integrating $q_t$ against a test volume form $\mu$ we obtain a sum of double integrals of the type studied in Theorem~\ref{thm:laplace} (with $\eps=2t$),
\begin{equation}\label{eq:heat-kernel-series}
    \int_M q_t\,d\mu = \sum_{k=0}^\infty t^k\iint_{M\times M} \frac{e^{-u(x,y)/2t}}{(4\pi t)^{d/2}} \Phi_k(x,y) q_0(x)dm(x)d\mu(y)\,.
\end{equation}
Let us explore what happens in the small-time limit. We first focus on the first term $k=0$. Set $r(dx,dy)=\Phi_0(x,y) q_0(x)m(dx)\mu(dy)$. Then the zeroth-order term in our Laplace formula as $t\to 0^+$ is
\[
    \int_M \Phi_0(x,x) q_0(x)d\mu(x)\,.
\]
Since the other terms $k\ge 1$ are $O(t)$, equating left-hand side and right-hand side in~\eqref{eq:heat-kernel-series} leads to
\[
    \Phi_0(x,x)=1,
\]
which is the well-known first information one typically obtains on the Hadamard coefficients~\cite[Chapter 3.2]{rosenberg1997}. In other words, $\Phi_0=1$ on $\Sigma$ and this implies that the tangential gradient vanishes, $\nabla \Phi_0=0$. 

Let us now expand to first-order in $t$ the first two terms in~\eqref{eq:heat-kernel-series}, substract the zeroth-order term, divide by $t$ and take the limit $t\to 0^+$. We obtain 
\[
    \int_M \partial_tq \,d\mu\Big|_{t=0} = 2\int_M \Big[- \frac18 \Deltat f + f \Big(\frac{3}{32}{\Rt}  - \frac{1}{8}R \Big)\Big] \,dm + \int_M\Phi_1 q_0d\mu\,,
\]
with $f(x,y)\coloneqq \Phi_0(x,y)q_0(x)dm(x)d\mu(y)/d\mt(x,y)$. 
Note the factor $2$ coming from $\eps=2t$. After some simplification the right-hand side can be written
\begin{equation}\label{eq:heat-laplace-integral}
    \int_M \big[\Delta q_0 - \bracket{K\nabla^N\!\Phi_0,\nabla q_0} + \big(-\frac14\Deltat\Phi_0-\frac12\div(K\nabla^N\!\Phi_0) + \Phi_1 - (\frac{\Rt}{16}+\frac{R}{4})\big) q_0\big]\,d\mu.
\end{equation}
Since $\partial_tq=\Delta q$ it implies the following equations on the diagonal $\Sigma$,
\[
    \nabla^N\!\Phi_0 =0\quad\text{and therefore}\quad \nablat\Phi_0 = 0,
\]
\[
    -\frac14\Deltat\Phi_0 + \Phi_1 =\frac{1}{16}\Rt+\frac{1}{4}R\,.
\]
As a matter of fact, more is known on these coefficients, for instance $\Phi_1(x,x)=\frac16 R$~\cite[Chapter 3]{rosenberg1997}. 

Let us now consider a more general situation, the Fokker--Planck equation
\begin{equation}\label{eq:fokker-planck}
\partial_tq =\Delta q + \nabla_{\!a} q + cq\,,
\end{equation}
where $a$ is a vector field on $M$, $\nabla$ the covariant derivative and $c$ a scalar field on $M$. The small-time asymptotics of~\eqref{eq:fokker-planck} were recently studied by Bilal in~\cite{bilal2020}, who showed that formally there exists a small-time expansion of the form~\eqref{eq:heat-kernel-series}, with of course adjusted coefficients $\Phi_k(x,y)$. Starting from~\eqref{eq:heat-laplace-integral} we obtain the following equations on the diagonal $\Sigma$,
\begin{gather*}
    \Phi_0(x,x)=1,\\
    \nabla\Phi_0=0,\\
    \nabla^N\!\Phi_0=-Ka,\\
    -\frac14\Deltat\Phi_0 +\Phi_1= c-\frac12\div(a)+\frac{1}{16}\Rt+\frac14 R\,.
\end{gather*}
We also note that similarly to the heat kernel, Bilal obtains more information on the coefficients, for instance $\Phi_1(x,x)= \frac16 R -\frac12\div(a)-\frac14\abs{a}^2+c$.

To conclude this example, we also look at the heat flow from a slightly different angle: given an initial $q_0$ on $M$ we consider the evolution flow 
\[
    \tilde q_t(y)=\int_M \frac{e^{-d^2(x,y)/4t}}{(4\pi t)^{d/2}} q_0(x)dm(x)\,,
\]
and want to see how it deviates from the heat equation solution $q$ defined by~\eqref{eq:heat-kernel-series}. Similar computations to the ones above give
\[
    \partial_t\tilde q|_{t=0} = \Delta q_0-\Big(\frac{1}{16}\Rt +\frac14 R\Big) q_0\,.
\]
We see that there are additional terms comprised of curvatures. 
While $R$ is the classical scalar curvature of $M$, $\Rt$ is not a traditionally studied Riemannian invariant since it requires to endow $M\times M$ with the Kim--McCann geometry. However we see that it shows up naturally in this very classical problem.

\subsection{Likelihood in Bayesian models with Gaussian priors}\label{SecBayesianApplication}

Bayesian modeling postulates a model of the observed data $y \in \R^d$ as being generated by  the combination of measurements on $x \in \R^k$ through a function $F\colon \R^k \to \R^d$ which can be nonlinear, as well as some model of noise. A standard model is $y = F(x) + \sqrt{\eps} n$ where $n$ is a normal Gaussian variable on $\R^d$ and $\eps$ a positive real parameter. The associated likelihood is given by
$
\mathbb{P}(dy|x) = (2\pi\eps)^{-d/2} e^{-\frac 1{2\eps} \abs{y- F(x)}^2}\,.
$
Let us consider the particular case where $k = d$ and $F$ is a $C^4$ diffeomorphism. It is a typical case where the probability is readily expressed under the Kim--McCann framework. In this case, the map is $y(x) = F(x)$ and the geometry on $\Sigma$ is flat whereas the ambient geometry is not. Therefore, several simplifications occur.
The geometric quantities read
\begin{enumerate}[(a)]
\item The metric on the product space $\R^d \times \R^d$ is $\gt = \frac 12 \begin{pmatrix} 0 & DF \\ DF& 0\end{pmatrix}$. It is non degenerate if $DF$ is non singular, and $ \mt =\abs{\det(DF)}$.
\item The metric on $\Sigma$ is the pull-back of the Euclidean metric on $\R^d$ by $F$, and  $m =  \mt$. Being the pull-back of a Euclidean metric, the curvature tensor of $g$ vanishes and $R = 0$.
\item $
      \Gammat_{ij}^k = \Gammat_{\bari \jb}^{\bark} = 0$ and it follows using Formulas \eqref{EqFormulaForh} that $h_{ij}^k = -\Gamma_{ij}^k = -\frac12 \frac{\partial [F^{-1}]^k}{\partial y^\kb}\frac{\partial^2 F^\kb}{\partial x^i\partial x^j}$.
As a consequence, the curvature tensor $\Rt _{i\jbar \bark \ell}$ vanishes and $\Rt  = 0$.
\item $ \Deltat f = 4 \operatorname{trace}([DF]^{-1} D^2_{xy}f)$.
\end{enumerate}
In this case, the geometric Laplace formula reads
\begin{multline*}
      \iint_{\Rd\times \Rd}\frac{e^{-u(x,y)/\eps}}{(2\pi\eps)^{d/2}}\,dr(x,y) = \int_{\Sigma} fdm \,+ \\ \eps\int_\Sigma\Big[-\frac 18\Deltat f+ \frac 14 \nablat_{\!H} f + \left( -\frac{1}{8}\bracket{H,H} + \frac{1}{24}\bracket{h,h}\right) f\Big] \,dm  + O(\eps^2)\,,
    \end{multline*}
    with $f=dr/d\mt$, and gives the joint law of the random variable $(x,y)$.
Other noise models such as multiplicative noise could be treated in a similar way, leading to different geometries.

\subsection{Fenchel--Young duality gap} \label{sec:fenchel-young-gap}
An example for which the formula can be made remarkably simple is the Fenchel--Young gap on Euclidean space. 
The corresponding ambient metric is flat and the metric on $\Sigma$ is a Hessian metric.
Indeed, consider a convex function $F\colon \Rd \to \R$ and $F^*\colon \Rd \to \R$ its Legendre--Fenchel transform defined by $F^*(y) = \sup_x \langle y , x \rangle - F(x)$. Thus the Fenchel--Young duality gap defined by  
$0\leq u(x,y) \coloneqq F(x) + F^*(y) - \langle x, y\rangle$ satisfies our assumptions under smoothness hypothesis on $F$. The set $\Sigma$ is defined by the graph of $D F$, the derivative of $F$. In such a case, the geometric quantities are
\begin{enumerate}[(a)]
\item The metric on the product space is $\gt = \frac 12 \begin{pmatrix} 0 & \operatorname{Id} \\ \operatorname{Id}& 0\end{pmatrix}$.
It is a flat metric and its Christoffel symbols and curvature tensor vanish.
\item The metric on $\Sigma$ is the Hessian metric of $F$, $g(x)(v,v) = \langle v,D^2 F(x) v \rangle$. 
\item For such a Hessian metric on the Euclidean space, one has
$
\Gamma^k_{ij} = F^{pk} F_{ijp}
$, where $F_{ijp}=\partial_{ijp}F$ and $F^{pk}$ denotes the inverse matrix of $F_{kp}$. Since  $\hGamma_{ij}^k = \Gamma_{ij}^k + h_{ij}^k = 0$, one has $h_{ij}^k = - \Gamma_{ij}^k = - F^{pk} F_{ijp}$. The curvature tensor reads
$R_{ijk\ell} = \frac 14 F^{mn}(F_{jkm}F_{i\ell n} - F_{j\ell m}F_{ikn})$.
\item The Laplacian on $X \times Y$ is $ \Deltat f = 4\operatorname{trace}(D^2_{xy}f) $.
\end{enumerate}
The Laplace formula reads
 \begin{multline*}
      \iint_{X\times Y}\frac{e^{-u(x,y)/\eps}}{(2\pi\eps)^{d/2}}f(x,y)\,d\mt (x,y) = \int_{\Sigma} fdm \,+ \\ \eps\int_\Sigma\Big[-\frac 18\Deltat f+ \frac 14 \nablat_{\!H} f+ \left( - \frac{1}{8}R  -\frac{1}{8}\bracket{H,H} + \frac{1}{24}\bracket{h,h}\right) f\Big] \,dm  + O(\eps^2)\,.
    \end{multline*}
\par{\textbf{Change of variable. }}
Now, we underline that, in practical cases, one is given a function $u$ on a manifold $X \times Y$. However, the Kim--McCann metric is defined via $u$ therefore it is not invariant to reparametrization.
As explained in the introduction, changes of variables can be used in order to simplify the first order formula, the simplest formulation being obtained via a change of variable that transforms $u$ to a gaussian in the standard case of Section \ref{sec:translation-invariant-cost}. However, since the vanishing set of $u$, $\Sigma$, is a $d$-dimensional manifold such a change of variable cannot be performed. 
Yet, for the Fenchel--Young gap, a natural change of variables is given by the map $DF\colon X \to Y$ which will be assumed a smooth diffeomorphism.
One can now consider the function $\tilde u\colon X \times X \to \R$ defined by 
\[
    \tilde u(x,x') = u(x,D F(x'))= F(x) + F^*(D F(x')) - \langle D F(x'),x\rangle.
\]
It is known as the Bregman divergence of $F$, see Examples~\ref{ex:bregmandiv} and~\ref{ex:fenchelyounggap}. The Kim--McCann metric is $$\gt(x,x') = \frac 12 \begin{pmatrix} 0 & D^2 F(x') \\ D^2 F(x')& 0\end{pmatrix}\,,$$ $\Sigma$ is the diagonal  $\{(x,x) \in X\times X\}$ and the metric on $\Sigma$ is again the Hessian metric of $F$. Note that the Kim--McCann metric does not depend on $x$ and thus its Christoffel symbols vanish. The interest of this change of variable is to set the second fundamental form to $0$. 
Therefore, in these coordinates, the Laplace formula is particularly simple
\begin{equation*}
      \iint_{X\times X}\frac{e^{-\tilde u(x,x')/\eps}}{(2\pi\eps)^{d/2}}f(x,x')\,d\mt (x,x') =  \int_{\Sigma} fdm - \frac\eps 8 \int_\Sigma \big(\Deltat f  + R  f\big) \,dm  + O(\eps^2)\,.
\end{equation*}
Note that for practical applications, the previous formula necessitates the knowledge of $D F$, implicitly given by the Fenchel--Young gap.

%% file: ack.tex
\section*{Acknowledgement}

F.L. was supported by ERC grant
NORIA and by the French government under management of Agence Nationale de la Recherche,
as part of the ``Investissements d’avenir'' program, reference ANR19-P3IA-0001 (PRAIRIE 3IA
Institute). F.L. also received funding from the European Research Council under the European Union’s Horizon 2020 research and innovation programme (Grant Agreement no. 866274).
\\
The work of F.-X. Vialard was partly supported by the Bézout Labex (New Monge Problems), funded by ANR, reference ANR-10-LABX-58.

%% file: main.bbl
\begin{thebibliography}{}

\bibitem[Alekseevsky et~al., 2009]{parakahler}
Alekseevsky, D.~V., Medori, C., and Tomassini, A. (2009).
\newblock Para-{K}\"{a}hler {E}instein metrics on homogeneous manifolds.
\newblock {\em C. R. Math. Acad. Sci. Paris}, 347(1-2):69--72.

\bibitem[Amari, 2016]{AmariBook}
Amari, S.-i. (2016).
\newblock {\em Information geometry and its applications}, volume 194 of {\em
  Applied Mathematical Sciences}.
\newblock Springer, [Tokyo].

\bibitem[Amari and Kumon, 1983]{AmariExponential}
Amari, S.-i. and Kumon, M. (1983).
\newblock Differential geometry of {E}dgeworth expansions in curved exponential
  family.
\newblock {\em Annals of the Institute of Statistical Mathematics}, 35:1--24.

\bibitem[Barndorff-Nielsen and Cox, 1989]{barndorff-nielsen}
Barndorff-Nielsen, O.~E. and Cox, D.~R. (1989).
\newblock {\em Asymptotic techniques for use in statistics}.
\newblock Monographs on Statistics and Applied Probability. Chapman \& Hall,
  London.

\bibitem[Bender and Orszag, 1999]{Bender1999}
Bender, C.~M. and Orszag, S.~A. (1999).
\newblock {\em Asymptotic Expansion of Integrals}, pages 247--316.
\newblock Springer New York, New York, NY.

\bibitem[Bilal, 2020]{bilal2020}
Bilal, A. (2020).
\newblock Small-time expansion of the {F}okker--{P}lanck kernel for space and
  time dependent diffusion and drift coefficients.
\newblock {\em Journal of Mathematical Physics}, 61(6):061517.

\bibitem[Bolthausen, 1986]{bolthausen}
Bolthausen, E. (1986).
\newblock Laplace approximations for sums of independent random vectors.
\newblock {\em Probability Theory and Related Fields}, 72(2):305--318.

\bibitem[Chavel, 1984]{Chavel1984EigenvaluesIR}
Chavel, I. (1984).
\newblock {\em Eigenvalues in Riemannian geometry}.
\newblock Academic Press.

\bibitem[Chen, 2014]{chen2014total}
Chen, B.-Y. (2014).
\newblock {\em Total mean curvature and submanifolds of finite type},
  volume~27.
\newblock World Scientific Publishing Company.

\bibitem[Chiappori et~al., 2010]{Chiappori-McCann-Nesheim-2010}
Chiappori, P.-A., McCann, R.~J., and Nesheim, L.~P. (2010).
\newblock Hedonic price equilibria, stable matching, and optimal transport:
  equivalence, topology, and uniqueness.
\newblock {\em Economic Theory}, 42(2):317--354.

\bibitem[Crane et~al., 2017]{Crane}
Crane, K., Weischedel, C., and Wardetzky, M. (2017).
\newblock The heat method for distance computation.
\newblock {\em Commun. ACM}, 60(11):90–99.

\bibitem[Cruceanu et~al., 1996]{parakahlerReview}
Cruceanu, V., Fortuny, P., and Gadea, P.~M. (1996).
\newblock A survey on paracomplex geometry.
\newblock {\em Rocky Mountain J. Math.}, 26(1):83--115.

\bibitem[Dajczer and Tojeiro, 2019]{dajczer2019submanifold}
Dajczer, M. and Tojeiro, R. (2019).
\newblock {\em Submanifold theory}.
\newblock Universitext. Springer, New York.
\newblock Beyond an introduction.

\bibitem[Galichon, 2016]{GalichonBook}
Galichon, A. (2016).
\newblock {\em Optimal transport methods in economics}.
\newblock Princeton University Press, Princeton, NJ.

\bibitem[Gangbo and McCann, 1995]{GangboMcCann1995}
Gangbo, W. and McCann, R.~J. (1995).
\newblock Optimal maps in {M}onge's mass transport problem.
\newblock {\em C. R. Acad. Sci. Paris S\'{e}r. I Math.}, 321(12):1653--1658.

\bibitem[Hwang, 1980]{hwang1980}
Hwang, C.-R. (1980).
\newblock {L}aplace's method revisited: Weak convergence of probability
  measures.
\newblock {\em Ann. Probab.}, 8:1177--1182.

\bibitem[Khan and Zhang, 2020]{KhanZhang2020}
Khan, G. and Zhang, J. (2020).
\newblock The {K}\"{a}hler geometry of certain optimal transport problems.
\newblock {\em Pure Appl. Anal.}, 2(2):397--426.

\bibitem[Kim and McCann, 2007]{kim2007continuity}
Kim, Y.-H. and McCann, R.~J. (2007).
\newblock Continuity, curvature, and the general covariance of optimal
  transportation.
\newblock {\em J. Eur. Math. Soc.}, 12.

\bibitem[Kim et~al., 2010]{Kim-McCann-Warren-calibrates}
Kim, Y.-H., McCann, R.~J., and Warren, M. (2010).
\newblock Pseudo-{R}iemannian geometry calibrates optimal transportation.
\newblock {\em Math. Res. Lett.}, 17(6):1183--1197.

\bibitem[Kolassa, 1997]{Kolassa1997}
Kolassa, J.~E. (1997).
\newblock {\em Multivariate Expansions}, pages 96--111.
\newblock Springer New York, New York, NY.

\bibitem[Ludewig, 2019]{Ludewig}
Ludewig, M. (2019).
\newblock Strong short-time asymptotics and convolution approximation of the
  heat kernel.
\newblock {\em Annals of Global Analysis and Geometry}, 55(2):371--394.

\bibitem[Ma et~al., 2005]{ma2005regularity}
Ma, X.-N., Trudinger, N.~S., and Wang, X.-J. (2005).
\newblock Regularity of potential functions of the optimal transportation
  problem.
\newblock {\em Archive for rational mechanics and analysis}, 177(2):151--183.

\bibitem[McCann, 1999]{McCann_line1999}
McCann, R.~J. (1999).
\newblock Exact solutions to the transportation problem on the line.
\newblock {\em R. Soc. Lond. Proc. Ser. A Math. Phys. Eng. Sci.},
  455(1984):1341--1380.

\bibitem[McCann, 2014]{McCann_glimpse2014}
McCann, R.~J. (2014).
\newblock A glimpse into the differential topology and geometry of optimal
  transport.
\newblock {\em Discrete Contin. Dyn. Syst.}, 34(4):1605--1621.

\bibitem[McCullagh, 1987]{GVK025155202}
McCullagh, P. (1987).
\newblock {\em Tensor methods in statistics}.
\newblock Monographs on statistics and applied probability. Chapman and Hall,
  London [u.a.].

\bibitem[Minakshisundaram and Pleijel, 1949]{minakshisundaram_pleijel_1949}
Minakshisundaram, S. and Pleijel, A.~. (1949).
\newblock Some properties of the eigenfunctions of the {L}aplace-operator on
  {R}iemannian manifolds.
\newblock {\em Canad. J. Math.}, 1:242--256.

\bibitem[Misner et~al., 2017]{MisnerThorneWheelerBook}
Misner, C.~W., Thorne, K.~S., and Wheeler, J.~A. (2017).
\newblock {\em Gravitation}.
\newblock Princeton, NJ: Princeton University Press, originally published by
  {W}. {H}. {Freeman} and {Company}, {New} {York} 1973 edition.

\bibitem[O'Neill, 1983]{ONeillBook}
O'Neill, B. (1983).
\newblock {\em Semi-{R}iemannian geometry}, volume 103 of {\em Pure and Applied
  Mathematics}.
\newblock Academic Press, Inc. [Harcourt Brace Jovanovich, Publishers], New
  York.
\newblock With applications to relativity.

\bibitem[Pal and Wong, 2016]{PalWongArbitrage2016}
Pal, S. and Wong, T.-K.~L. (2016).
\newblock The geometry of relative arbitrage.
\newblock {\em Math. Financ. Econ.}, 10(3):263--293.

\bibitem[Pal and Wong, 2018]{PalWong_new_information_geometry2018}
Pal, S. and Wong, T.-K.~L. (2018).
\newblock Exponentially concave functions and a new information geometry.
\newblock {\em Ann. Probab.}, 46(2):1070--1113.

\bibitem[Peeters, 2007a]{peeters2007}
Peeters, K. (2007a).
\newblock {C}adabra: a field-theory motivated symbolic computer algebra system.
\newblock {\em Computer Physics Communications}, 176(8):550--558.

\bibitem[Peeters, 2007b]{peeters2007arXiv}
Peeters, K. (2007b).
\newblock Introducing {C}adabra: A symbolic computer algebra system for field
  theory problems.
\newblock {\em arXiv preprint hep-th/0701238}.

\bibitem[Peeters, 2018]{cadabra}
Peeters, K. (2018).
\newblock Cadabra2: computer algebra for field theory revisited.
\newblock {\em Journal of Open Source Software}, 3(32):1118.

\bibitem[Reid, 1988]{ReidSaddlepoint}
Reid, N. (1988).
\newblock {Saddlepoint Methods and Statistical Inference}.
\newblock {\em Statistical Science}, 3(2):213 -- 227.

\bibitem[Rosenberg, 1997]{rosenberg1997}
Rosenberg, S. (1997).
\newblock {\em The {L}aplacian on a {R}iemannian Manifold: An Introduction to
  Analysis on Manifolds}.
\newblock London Mathematical Society Student Texts. Cambridge University
  Press.

\bibitem[Santambrogio, 2015]{santambrogio2015optimal}
Santambrogio, F. (2015).
\newblock Optimal transport for applied mathematicians.
\newblock {\em Birk{\"a}user, NY}, 55(58-63):94.

\bibitem[Shun and McCullagh, 1995]{ShunMcCullagh}
Shun, Z. and McCullagh, P. (1995).
\newblock Laplace approximation of high dimensional integrals.
\newblock {\em Journal of the Royal Statistical Society. Series B
  (Methodological)}, 57(4):749--760.

\bibitem[Simons, 1968]{Simons1968}
Simons, J. (1968).
\newblock Minimal varieties in riemannian manifolds.
\newblock {\em Ann. of Math. (2)}, 88:62--105.

\bibitem[Strawderman, 2000]{ReviewStrawderman}
Strawderman, R.~L. (2000).
\newblock Higher-order asymptotic approximation: Laplace, saddlepoint, and
  related methods.
\newblock {\em Journal of the American Statistical Association},
  95(452):1358--1364.

\bibitem[Tierney and Kadane, 1986]{TierneyKadane}
Tierney, L. and Kadane, J.~B. (1986).
\newblock Accurate approximations for posterior moments and marginal densities.
\newblock {\em Journal of the American Statistical Association},
  81(393):82--86.

\bibitem[Villani, 2008]{villani2008optimal}
Villani, C. (2008).
\newblock {\em Optimal transport: old and new}, volume 338.
\newblock Springer Science \& Business Media.

\bibitem[Wong, 2001]{WongBook}
Wong, R. (2001).
\newblock {\em Asymptotic Approximations of Integrals}.
\newblock Society for Industrial and Applied Mathematics.

\bibitem[Wong, 2018]{WongLogDiv2018}
Wong, T.-K.~L. (2018).
\newblock Logarithmic divergences from optimal transport and {R}\'{e}nyi
  geometry.
\newblock {\em Inf. Geom.}, 1(1):39--78.

\bibitem[Wong and Yang, 2022]{wong2021pseudoriemannian}
Wong, T.-K.~L. and Yang, J. (2022).
\newblock Pseudo-{R}iemannian geometry encodes information geometry in optimal
  transport.
\newblock {\em Inf. Geom.}, 5(1):131--159.

\bibitem[Xin, 2018]{xin2018minimal}
Xin, Y. (2018).
\newblock {\em Minimal submanifolds and related topics}, volume~16.
\newblock World Scientific.

\end{thebibliography}
